\documentclass{article}

\usepackage{amsmath}
\usepackage{amssymb}
\usepackage{amsthm}
\usepackage[latin1]{inputenc}

\usepackage{graphicx}
\usepackage{epstopdf}

\usepackage{ifthen}

\renewcommand{\div}{\operatorname{div}}

\newcommand{\osc}{\operatorname{osc}}

\newcommand{\Tt}{{\mathbb{T}}}

 \newcommand{\Rr}{\mathbb R}

 \newcommand{\af}{\alpha}
 
 \newcommand{\ep}{\epsilon}
\newcommand{\be}{\beta}
 \newcommand{\ga}{\gamma}
 
 \newcommand{\de}{\delta}
 
  \newcommand{\lam}{\lambda}
 \newcommand{\te}{\theta}

\renewcommand{\div}{\operatorname{div}}
\newcommand{\diver}{\operatorname{div}}

\newcommand{\tr}{\operatorname{Tr}}

\newcommand{\bx}{{\bf x}}

\newcommand{\br}{{\bf r}}

\newcommand{\bv}{{\bf v}}

\newtheorem{teo}{Theorem}[section]
\newtheorem{Definition}{Definition}[section]
\newtheorem{Lemma}{Lemma}[section]
\newtheorem{Corollary}{Corollary}[section]
\newtheorem{Proposition}{Proposition}[section]
\newtheorem{Remark}{Remark}[section]
\newtheorem{Assumption}{A}

\begin{document}

\title{Time-dependent mean-field games in the subquadratic case}
\author{Diogo
  A. Gomes\footnote{Center for Mathematical Analysis, Geometry, and Dynamical Systems, Departamento de Matem\'atica, Instituto Superior T\'ecnico, 1049-001 Lisboa, Portugal, 
King Abdullah University of Science and Technology (KAUST), CSMSE Division , Thuwal 23955-6900. Saudi Arabia. e-mail: dgomes@math.ist.utl.pt, and
KAUST SRI, Uncertainty Quantification Center in Computational Science and Engineering.}, 
Edgard Pimentel
\footnote{Center for Mathematical Analysis, Geometry, and Dynamical Systems, Departamento de Matem\'atica, Instituto Superior T\'ecnico, 1049-001 Lisboa, Portugal. e-mail: epiment@math.ist.utl.pt}, 
H\'ector S\'anchez-Morgado.  
\footnote{Instituto de Matem\'aticas, Universidad Nacional Aut\'onoma
  de M\'exico. e-mail: hector@matem.unam.mx}}

\date{\today} 

\maketitle

\begin{abstract}
In this paper we consider time-dependent mean-field games with subquadratic Hamiltonians and power-like local dependence on the measure. We establish existence of classical solutions under a certain set of conditions depending on both the growth of the Hamiltonian and the dimension. This is done by combining 
regularity estimates for the Hamilton-Jacobi equation based on the
Gagliardo-Nirenberg interpolation inequality with polynomial estimates for the Fokker-Planck equation. This technique improves substantially 
the previous results on the regularity of time-dependent mean-field games. 
\end{abstract}

\thanks{
D. Gomes was partially supported by CAMGSD-LARSys through FCT-Portugal and by grants
PTDC/MAT-CAL/0749/2012,  UTA-CMU/MAT/0007/2009 
PTDC/MAT/114397/2009, and
UTAustin/MAT/0057/2008.

E. Pimentel was supported by CNPq-Brazil, grant GDE/238040/2012-7.}

\section{Introduction}

The mean-field games framework
is a class of methods inspired by ideas in statistical physics which
aims at understanding differential games with infinitely many indistinguishable players. 
Since the pioneering works
 \cite{C1, C2, ll1, ll2, ll3, ll4}
this area has known an intense research activity, see, for instance, the recent surveys \cite{llg2}, \cite{cardaliaguet}, \cite{achdou2013finite}, or \cite{GS}, as well as the video lectures by P-L. Lions \cite{LCDF,LIMA},
and the references therein. In particular,
a number of problems have been worked out in detail by various authors, 
including numerical methods \cite{lst}, \cite{DY}, \cite{CDY}, 
applications
in economics \cite{llg1}, \cite{GueantT} and environmental
policy \cite{lst}, finite state problems \cite{GMS}, \cite{GMS2}, \cite{GF}, 
explicit models \cite{Ge}, \cite{NguyenHuang},
obstacle-type problems \cite{GPat}, extended mean-field games \cite{GPatVrt},
\cite{GVrt}, probabilistic methods \cite{Carmona1}, \cite{Carmona2}, long-time behavior
\cite{CLLP}, \cite{Cd1}, weak solutions
to mean-field games  \cite{porretta}, \cite{Cd2}, rigorous justification of mean-field games systems \cite{BFl}, \cite{KLY},
to mention just a few.
 
A model mean-field game problem is the system
\begin{equation}
\label{eq:smfg0}
\begin{cases}
-u_t+H(x, Du)=\Delta u +g(m)\\
m_t-\div(D_pH m)=\Delta m.
\end{cases}
\end{equation}
As usual, $H$ and $g$ satisfy various conditions as detailed in Section \ref{ma}.
In many applications the boundary conditions for
the previous equations are the, so called, initial-terminal boundary conditions:
\begin{equation}
\label{itvp}
\begin{cases}
u(x,T)=u_0(x)\\
m(x,0)=m_0(x),
\end{cases}
\end{equation}where $T>0$ is a fixed terminal instant. To avoid additional difficulties, we will consider in this paper spatially 
periodic solutions.
That is, $u$ and $m$ are regarded as functions with domain $\Tt^d\times [0,T]$, where $\Tt^d$ is
the $d$-dimensional torus.

%
%


The objective of this paper 
is to obtain conditions under which existence of solutions to \eqref{eq:smfg0} under the initial-terminal conditions \eqref{itvp}
 can be established. 
This fundamental question
 has been addressed in several ways by various authors. The first result
on existence of solutions appeared in 
\cite{ll2} where the authors consider weak solutions to \eqref{eq:smfg0}-\eqref{itvp}.  Weak solutions for the planning problem were considered in \cite{porretta}. 
For quadratic Hamiltonians the existence of smooth solutions to \eqref{eq:smfg0}-\eqref{itvp} was established in \cite{CLLP}. The proof
in that paper
uses strongly 
the quadratic structure of the Hamiltonian through the Hopf-Cole transformation and does not extend (except perhaps in very specific perturbation regimes) 
to general Hamiltonians with quadratic growth. By using
the classical results from \cite{CrAm} combined with elementary techniques, \cite{LIMA} verifies that mean-field games with Hamiltonians satisfying quadratic or subquadratic growth conditions and with the model non-linearity $g(m)=m^\alpha$
admit classical solutions under the constraint
$0<\alpha<\frac{2}{d-2}$, if $d>2$ and any $\alpha>0$ if $d=2$. 
A second result also from \cite{LIMA} concerns subquadratic Hamiltonians $H$
behaving at $\infty$ like $|p|^\gamma$, $\gamma<1+\frac{1}{d+1}$, 
and $g(m)=m^\alpha$. In this case existence of classical solutions is guaranteed for
$0<\alpha<\infty$.
In the stationary case the first result on existence of solutions was obtained in \cite{ll1}. Existence of smooth solutions was addressed in \cite{GM} (see also \cite{GIMY} for a related problem), \cite{GPM1}, and \cite{GPatVrt}. The results in \cite{e}, although not mentioning mean-field games, imply in fact existence of smooth solutions for certain stationary first order mean-field games. 

In this paper, as in \cite{LIMA}, we
consider also the model non-linearity $g(m)=m^\alpha$, and
improve and extend the previous
results substantially for Hamiltonians with subquadratic growth. In a companion paper \cite{GPM3}
we address superquadratic Hamiltonians which require different techniques. 

The specific assumptions under which we work, A\ref{ah}-A\ref{alphaimp}, are discussed in 
Section \ref{ma}. Under these assumptions, 
our main result is 
\begin{teo}\label{teo:intro1}
Suppose A\ref{ah}-A\ref{alphaimp} hold.
 Then
there exists a classical solution $(u,m)$ to \eqref{eq:smfg0}
 under the initial-terminal conditions \eqref{itvp}.
\end{teo}

Theorem \ref{teo:intro1} applied to mean-field games with Hamiltonians behaving at $\infty$ like $|p|^\gamma$, $1+\frac{1}{d+1}<\gamma<2$, and $g(m)=m^\alpha$ yields existence of smooth solutions for $0<\alpha<\alpha_{\ga,d}$ with $\alpha_{\ga,d}>\frac{2}{d-2},$
given in Assumption A\ref{alphaimp}.

In order to prove Theorem \ref{teo:intro1} we consider a regularization of \eqref{eq:smfg0} by replacing $g(m)$ by the nonlocal operator $g_\epsilon(m)=\eta_\epsilon*g(\eta_\epsilon*m)$, 
where $\eta_\epsilon$ is a standard mollifying kernel, which in particular is symmetric. This yields the system
\begin{equation}
\label{eq:smfg}
\begin{cases}
-u_t^\epsilon+H(x, Du^\epsilon)=\Delta u^\epsilon +g_\epsilon(m^\epsilon)\\
m_t^\epsilon-\div(D_pH m^\epsilon)=\Delta m^\epsilon.  
\end{cases}
\end{equation}
We use the convention $g_0=g$. The existence of solutions to \eqref{eq:smfg} follows from standard arguments using some of the ideas in \cite{cardaliaguet}. A detailed proof of this result is discussed in \cite{PIM}.

The proof proceeds by establishing a new class of polynomial estimates for $m^\ep$, which are combined with upper bounds for $u^\ep$. 

\begin{teo}\label{cpp2}
Let $(u^\epsilon,m^\epsilon)$ be a solution of \eqref{eq:smfg}. 
Suppose $m^\epsilon \in L^\infty([0,T], L^{\be_0}(\Tt^d))$, $\be_0\geq 1$.
Assume that $p\;>\;\frac{d}{2}$, let $q$ be the conjugate exponent
and $r=\frac{1}{\kappa}$, where
\begin{equation}
\label{bbkappa}
\kappa=\frac{d+2q-dq}{q[(\theta-1)d+2]}.
\end{equation} Then, 
\begin{align*}
\int_{\Tt^d}\left(m^\epsilon\right)^{\be_n}\left(\tau,x\right)dx
\leq C+C\left\|\left|D_{p}H\right|^{2}\right\|_{L^{r}\left(0,T;L^{p}\left(\Tt^d\right)\right)}^{r_n}
\end{align*}where
\begin{equation}\label{eq:eq221}
r_{n}=r\frac{\theta^{n}-1}{\theta-1},
\end{equation}$\theta<1$ and $\be_n=\theta^n\be_0$.
\end{teo}The proof of Theorem \ref{cpp2} is presented in Section \ref{elep}. The key upper bounds for $u^\ep$ are given by:

\begin{Lemma}\label{l8421}
Let $(u^\epsilon,m^\epsilon)$ be a solution of \eqref{eq:smfg} and assume that
A\ref{ah}-\ref{coerca} hold. Let $a,\;b>1$ be such that 
$$\frac{d}{2}<\frac{b(a-1)}{a}.$$Then there exists $C>0$ such
that  
\[
\|u^\epsilon\|_{L^\infty(0,T;L^\infty(\Tt^d))}\leq C+C\|g_\ep(m)\|_{L^a(0,T;L^b(\Tt^d))}.
\]
\end{Lemma}

Lemma \ref{l8421} is proved in Section \ref{pub}. Using the Gagliardo-Nirenberg interpolation Theorem we obtain:

\begin{teo}\label{c92}
Let $(u^\ep,m^\ep)$ be a solution of \eqref{eq:smfg} and assume that
A\ref{ah}-\ref{asbdpph} hold. 
For $1<p, r<\infty$ there are positive constants $c$ and $C$ such that 
\[\|D^2u^\ep\|_{L^r(0,T;L^p(\Tt^d))} \le c\|g_\ep(m^\ep)\|_{L^r(0,T;L^p(\Tt^d))}+c\|u^\ep\|_{L^{\infty}(0,T;L^{\infty}(\Tt^d))}^\frac{\ga}{2-\ga}+C.\]
\end{teo}The proof of Theorem \ref{c92} is presented in Section \ref{rsbqc}. The proof of Theorem \ref{teo:intro1} follows by combining the estimates in Theorem \ref{cpp2}, Lemma \ref{l8421} and Theorem \ref{c92}, which, together with Gagliardo-Nirenberg Theorem yield a key bound for $\|Du^\ep\|_{L^r(0,T;L^p(\Tt^d))}$. After this, a number of additional estimates, as we outline next, end the proof.



The paper is organized as follows: Section \ref{ma} introduces the main assumptions and presents a model Hamiltonian satisfying those. Preliminary estimates are discussed in Sections \ref{apes} to \ref{best}. 
In Section \ref{elsoe} we obtain second order estimates that enable us to address in
Section \ref{igain} the regularity of solutions to the Fokker-Planck equation. 
In section \ref{sec10} we prove the main estimate by combining the results in Section \ref{rsbqc} with the ones from Section \ref{igain}. 
In section \ref{lrsbh} we show that $u$ is Lipschitz by an application of the non-linear adjoint method
introduced in \cite{E3}.
Once Lipschitz regularity is obtained for the Hamilton-Jacobi equation, a number of estimates on the Fokker-Planck equation can be established which improve significantly
the regularity for $m$. These are discussed in Section \ref{irfp}. Finally, in Section \ref{frhct} we apply the Hopf-Cole transformation 
to obtain further a priori bounds. Namely in Theorem \ref{mLip} we obtain that $\ln m^\epsilon$ is Lipschitz. All the previous estimates are uniform in $\epsilon$. 
Therefore, from this, we can then bootstrap the smoothness of solutions and pass to the limit as $\epsilon\to 0$. This is done in Section \ref{sec17} where we finally prove Theorem \ref{teo:intro1}.

The authors thank P. Cardaliaguet, P-L. Lions, A. Porretta and P. Souganidis for very useful comments and suggestions.

\section{Main assumptions}\label{ma}

In this section we discuss the main assumptions used throughout the paper. We have tried to 
work under fairly general hypothesis which cover a wide range of interesting problems. 
We end the section with an example which illustrates some of the applications of our results. 

\subsection{Assumptions}\label{subsec:assumptions}

We start by setting the general hypothesis on $H$ and $g$.

\begin{Assumption}
\label{ah}
The Hamiltonian
$H:\Tt^d\times \Rr^d \to \Rr$, $d>2$, is smooth and:
\begin{enumerate}
\item 
for fixed $x$, $p\mapsto H(x,p)$ is a strictly convex function;  
\item satisfies the coercivity condition
\[
\lim_{|p|\to \infty} \frac{H(x,p)}{|p|}=+\infty,   
\] 
and without loss of generality we suppose further that $H(x,p)\geq 1$. 
\end{enumerate}
Futhermore,  $(u_0, m_0)\in C^\infty(\Tt^d)$ with $m_0\geq 0$, and $\int_{\Tt^d} m_0=1$. 
\end{Assumption}

\begin{Assumption}
\label{ag}
$g:\Rr_0^+\to \Rr$ is a non-negative increasing function. 
\end{Assumption}

From the previous hypothesis it follows that $g(z)=G'(z)$ for some
convex increasing function $G:\Rr^+_0\to\Rr$.  

We define the Legendre transform of $H$ by
\begin{equation}
\label{hatele}
L(x,v)=\sup_p \left( -p\cdot v-H(x,p)\right).  
\end{equation}
Then if we set
\begin{equation}
\label{ele}
\hat L(x,p)=D_pH(x,p) p-H(x,p),  
\end{equation}
by standard properties of the Legendre transform
$
\hat L(x,p)=L(x,-D_pH(x,p)).
$

\begin{Assumption}
\label{aele}
For some $c, C>0$
\[
\hat L(x,p)\geq c H(x,p)-C.
\]
\end{Assumption}

Though various results can be proved under 
minimalistic assumptions
such as A\ref{ah}-\ref{aele}, see Section \ref{apes},
to proceed further one needs 
the additional condition $g\geq 0$. For definiteness and convenience
we take $g$ to be a power non-linearity:
\begin{Assumption}
\label{ag2}
$g(m)=m^\alpha$, for some $\alpha>0$.
\end{Assumption}

The general case with $x$ dependence $g(x,m)$ could be addressed similarly without any major changes as long as suitable growth conditions and bounds on $g$ are imposed.  

%

\begin{Assumption}
\label{strong}
$H$
satisfies 
$
|D_xH|, |D^2_{xx}H|\leq CH+C, 
$
and, for any symmetric matrix $M$, and any $\delta>0$ there exists $C_\delta$ such that 
\[
\tr(D^2_{px} H M) \leq \delta \tr(D^2_{pp} H M^2)+C_{\delta} H. 
\]
\end{Assumption}

Note that since we assume $H\geq 1$ we can replace the inequality in
the previous Assumption by $|D_xH|, |D^2_{xx}H|\leq \tilde C H$, for
some constant $\tilde C$.  

%

\begin{Assumption}
\label{bcc}
We have $m_0\geq \kappa_0$
for some $\kappa_0\in \Rr^+$. 
\end{Assumption}


The next group of hypothesis concerns subquadratic growth. Since the results in \cite{LIMA}
yield existence of solution for Hamiltonians with growth at infinity like $|p|^\gamma$
for $\gamma<1+\frac{1}{d+1}$ we only consider larger growth exponents.

\begin{Assumption}
\label{coerca}
$H$ satisfies the sub-quadratic growth condition
$
H(x,p)\leq C|p|^\ga +C
$, for some $1+\frac{1}{d+1}<\ga<2$.
\end{Assumption}

\begin{Assumption}
\label{dphs}
$D_pH$ satisfies the following growth condition
$
\left|D_pH\right|\leq C\left|p\right|^{\ga-1}+C
$, for some $1+\frac{1}{d+1}<\ga<2$.
\end{Assumption}

\begin{Assumption}
\label{asbdpph}
$H$ satisfies
$\left|D^2_{xp}H\right|^2\leq CH$
and, for any symmetric matrix $M$
\[
\left|D^2_{pp}HM\right|^2\leq  C \tr(D^2_{pp} H M M). 
\]
\end{Assumption}

The second assertion in Assumption A\ref{asbdpph} ensures the existence of a uniform upper bound for the eigenvalues of $D^2_{pp}H$.
%

For convenience we present here a technical lemma:
\begin{Lemma}\label{techlemma}
Let $d>2$. There exists $\af_{\ga,d}$, with
\begin{align*}
\af_{\ga,d}>\frac{-4(-4+\ga)^2(-1+\ga)\ga^2+2d(-4+(-2+\ga)\ga)(-4+(-4+\ga)(-2+\ga)\ga)}{(-2+d)(-4+\ga)(-1+\ga)\ga(-2(-4+\ga)\ga+d(-4+(-2+\ga)\ga))},
\end{align*}
such that, for $\alpha<\af_{\ga,d}$ there are $\lam,\,\zeta,\,\upsilon,\,a_\upsilon,\,b_\upsilon,\,r,\,\tilde{r},\,p,\,\tilde{p},\,\theta,\,F$ and $G$ satisfying \eqref{eq:1sec10}-\eqref{eq:eqG}, \eqref{eq:5sec10}-\eqref{eq:6sec10} and \eqref{eq:7sec10}-\eqref{eq:8sec10}.
\end{Lemma}
\begin{proof}
The Lemma is proved by using the symbolic software Mathematica. For the details, we refer the reader to \cite{PIM}.
\end{proof}

Notice that, for $d>2$ and $1<\ga<2$ one has

\begin{align*}
\frac{-4(-4+\ga)^2(-1+\ga)\ga^2+2d(-4+(-2+\ga)\ga)(-4+(-4+\ga)(-2+\ga)\ga)}{(-2+d)(-4+\ga)(-1+\ga)\ga(-2(-4+\ga)\ga+d(-4+(-2+\ga)\ga))}>\frac
{2}{d-2},
\end{align*}

\begin{Assumption}
\label{alphaimp}
The exponent $\alpha$ is such that 
$0<\alpha<\alpha_{\ga,d}$ 
\end{Assumption}

\subsection{A model Hamiltonian}\label{subsec:sec22}

Next we discuss an example satisfying the Assumptions introduced previously. Consider the following Hamiltonian
\[
H_s(x,p)=a(x)\big(1+|p|^2\big)^\frac{\ga}{2}+V(x),
\]where $a, V\in C^{2}(\Tt^d)$ with $a, V >0$. Thus
$\left(1+|p|^2\right)^\frac{\ga}{2}\le C H_s(x,p)$. 
It is clear that $H$ satisfies A\ref{ah}.

Using \eqref{ele} we have
$\hat
L_s(x,p)=a(x)\left((\gamma-1)|p|^2-1\right)\left(1+\left|p\right|^2\right)^\frac{\ga-2}{2}-V(x)$, 
from which it follows A\ref{aele}. 
Observe that 
$H_s(x,p)$, $|D_xH_s(x,p)|$, $|D^2_{xx}H_s(x,p)|\leq
C \left(1+|p|^2\right)^{\frac\ga 2}$ 
and 
$\left|D_pH_s\right|=a(x)\ga\left(1+|p|^2\right)^{\frac{\ga-2}{2}}|p|$.
Hence, the first part of 
A\ref{strong}, A\ref{coerca} and A\ref{dphs} are clearly satisfied. 

Now notice that 
\[
D_{p_ip_j}^2H_s=a(x)\ga(\ga-2)(1+|p|^2)^\frac{\ga-4}{2}p_ip_j \quad i\ne j,
\]
and
\[ D_{p_jp_j}^2H_s= \ga a(x)
\left((\ga-2)(1+|p|^2)^\frac{\ga-4}{2}p_j^2+(1+|p|^2)^\frac{\ga-2}{2}\right). 
\]
Hence, for any symmetric matrix $M$
$
  \tr(D^2_{pp}H_sM^2)
\geq C (1+|p|^2)^\frac{\ga-2}{2}|M|^2 
$, 
since $\gamma>1$.
Notice that 
\begin{align*}
&\left|\tr\left(D^2_{xp}H_sM\right)\right|\leq \ga\left(1+\left|p\right|^2\right)^\frac{\ga-2}{2}\left|Da\right| |pM|
 \leq \delta \tr(D^2_{pp}H_sM^2)+C H_s. 
\end{align*}
This shows that the second part of A\ref{strong} also holds. 

In order to verify that $H_s$ satisfies A\ref{asbdpph}, notice that 
$|D_{xp}^2H_s|^2\le C\big(1+|p|^2\big)^{\ga-1}.$
Since $\ga-1<\dfrac\ga 2$ for $\ga<2$, the first part of the Assumption is
verified. For the second part 
observe that  
\begin{align*}
  \big|D_{pp}^2H_sM\big|^2& \le C(1+|p|^2)^{\ga-4}\big(|p|^2|Mp|^2+
(1+|p|^2)^2|M|^2+(1+|p|^2)|Mp|^2\big)\\
&\leq C^* (1+|p|^2)^{\ga-2}|M|^2\leq C (1+|p|^2)^{\frac{\ga-2}{2}}|M|^2, 
\end{align*}
since $1<\ga<2$. Thus $H_s$ satisfies also the second part of A\ref{asbdpph}.

\section{Lax-Hopf estimate}
\label{apes}

In this section we establish various estimates for the solutions of
\eqref{eq:smfg}. These  follow from  
the stochastic optimal control representation of solutions of Hamilton-Jacobi equations. \emph{}

We start by recalling the stochastic control representation for solutions 
to the first equation in \eqref{eq:smfg}, which we call the stochastic
 Lax-Hopf formula. For that, let $(u^\epsilon,m^\epsilon)$ be a
 solution of \eqref{eq:smfg}. 
Then, see for instance \cite{FS}, $u^\ep$ is the value function for
the following stochastic optimal control problem 
\begin{equation}
\label{soc}
u^\epsilon(x,t)=\inf_{\bv} E\int_t^T L(\bx(s),\bv(s))+g_\epsilon(m^\epsilon)(\bx(s),s)ds + u^\epsilon(\bx(T),T), 
\end{equation}
where $L$ is given by \eqref{hatele}, and the infimum is taken over
all bounded and progressively measurable controls $\bv$, 
$
d\bx=\bv ds+\sqrt{2} dW_s$, 
where $\bx(t)=x$, and  $W_s$ is a $d$-dimensional Brownian motion. The estimates that we
discuss now can be regarded  
as a consequence of this optimal control representation formula.

\begin{Proposition}
\label{plb}
Suppose A\ref{ah} holds. Let $(u^\epsilon,m^\epsilon)$ be a solution
to \eqref{eq:smfg}.  
Then, for any smooth vector field $b:\Tt^d\times(t,T)\to\Rr^d$,
and any solution to 
\begin{equation}
\label{diflaw2}
\zeta_s+\div(b\zeta)=\Delta \zeta, 
\end{equation}
with $\zeta(x,t)=\zeta_0$ 
we have
 the following upper bound: 
\begin{align}
\label{lhe2}
\int_{\Tt^d} u^\epsilon(x,t)\zeta_0(x) dx \le
&\int_t^T\int_{\Tt^d}\bigl(L(y,b(y,s))
+g_\epsilon(m^\epsilon)(y,s)\bigr)\zeta(y,s) dy ds\\\nonumber
&+\int_{\Tt^d} u^\epsilon(y, T) \zeta(y,T). 
\end{align} 
\end{Proposition}
\begin{proof}
Multiply the first equation in \eqref{eq:smfg} by $\zeta$ and multiply
\eqref{diflaw2} by $u^\epsilon$. Subtract these equations and integrate in
$\Tt^d$ to conclude that 
\[
-\frac{d}{dt}\int_{\Tt^d}u^\epsilon\zeta dx=\int_{\Tt^d}(-b(x,t)\cdot
Du^\epsilon-H(x,Du^\epsilon) 
+g_\epsilon(m^\epsilon))\zeta dx. 
\]
Using the inequality
$
L(x, b)\geq -H(x, p)-p\cdot b, 
$
we obtain the result. 
\end{proof}

A natural choice in the previous Proposition is $b=0$, and for
$\zeta_0$ either the Lebesgue measure 
or the measure $m_0$. A further choice of $b$ is of course
$b=-D_pH(x,Du^\epsilon)$, the optimal feedback control for \eqref{soc},  
and $\zeta_0=\delta_{x_0}$. This last choice makes it possible to
establish pointwise estimates and will be used in Section \ref{pub}.  

\begin{Corollary}
Suppose A\ref{ah} holds.
Let $(u^\epsilon,m^\epsilon)$ be a solution to \eqref{eq:smfg}.
We have the following two estimates: 
\begin{equation}
\label{etreze}
\int_{\Tt^d}u^\epsilon(x,0)m_0 dx\le
CT+\int_0^T\int_{\Tt^d}g_\epsilon(m^\epsilon)(x,t)\mu(x,t) dxdt 
+\int_{\Tt^d} u^\epsilon(x,T) \mu(x,T)dx,
\end{equation}
where $\mu(x,t)$ is the solution to the heat equation with $\mu(x,0)=m_0$, 
and
\begin{equation}
\label{ecatorze}
\int_{\Tt^d} u^\epsilon(x, 0)dx\leq CT +\int_0^T \int_{\Tt^d}
g_\epsilon(m^\epsilon)(x,t)dxdt+\int_{\Tt^d} u^\epsilon(x,T) dx, 
\end{equation}
\end{Corollary}
\begin{proof}
Both estimates follow from the choice $b=0$ in Proposition \ref{plb}.
With $\zeta_0=m_0$, for estimate \eqref{etreze} and $\zeta_0=1$ for
\eqref{ecatorze}.  
\end{proof}

\section{First order estimates}
\label{lle}

In this section we recover (and improve slightly) first order estimate
from \cite{ll2, ll3} for the regularized problem. These are also considered in \cite{CLLP}.
Recall that for a function $f$ we define
\[
\osc f=\sup f-\inf f. 
\]

\begin{Proposition}
\label{pehm}
Assume A\ref{ah}-\ref{aele} hold. 
Let $(u^\epsilon ,m^\epsilon)$ be a solution of \eqref{eq:smfg}. Then
\begin{equation}
\label{ihm}
\int_0^T \int_{\Tt^d} c H(x, D_xu^\epsilon) m^\epsilon
+G(\eta_\epsilon*m^\epsilon) dx dt 
\leq CT+C \osc u^\epsilon(\cdot, T), 
\end{equation}
where $G'=g$.
\end{Proposition}
\begin{proof}
We have
\[
-\frac{d}{dt}\int_{\Tt^d}u^\epsilon m^\epsilon dx+\int_{\Tt^d}(H-D_pH
D_xu^\epsilon)m^\epsilon dx =\int_{\Tt^d}m^\epsilon g_\epsilon(m^\epsilon)dx.
\]
Because $\eta_\epsilon(y)=\eta_\epsilon(-y)$, the previous computation yields,  using Assumption A\ref{aele}, 
\begin{align*}
&c \int_0^T\int_{\Tt^d} H(x,Du^\epsilon) m^\epsilon dx\,dt\leq 
\int_0^T\int_{\Tt^d} (D_pH D_xu^\epsilon-H) m^\epsilon dx\,dt=\\
&-\int_0^T\int_{\Tt^d} \eta_\epsilon*m^\epsilon
g(\eta_\epsilon*m^\epsilon)dx+\int_{\Tt^d}\left( u^\epsilon(x,0)
  m^\epsilon(x, 0)- u^\epsilon(x,T) m^\epsilon(x, T)\right)dx.  
\end{align*}
We now use the estimate \eqref{etreze} to conclude
that
\begin{align*}
c \int_0^T\int_{\Tt^d}H(x,Du^\epsilon)m^\epsilon dx\,dt \leq 
&CT+\int_{\Tt^d} u^\epsilon(x,T)(\mu(x,T)- m^\epsilon(x,T))dx\\
&+\int_0^T\int_{\Tt^d} g(\eta_\epsilon*m^\epsilon) \eta_\epsilon*(\mu -m^\epsilon) dx\,dt,
\end{align*}where $\mu$ is the solution of the heat equation with $\mu(x,0)=m_0(x)$.

By Assumption A\ref{ag}, there exists a convex function $G$ such that
$G'(z)=g(z)$. Therefore we have 
$
g(\eta_\epsilon*m^\epsilon)\eta_\epsilon*(\mu-m^\epsilon) \leq
G(\eta_\epsilon*\mu) -G(\eta_\epsilon*m^\epsilon), 
$
hence
\begin{align*}
c \int_0^T\int_{\Tt^d} H(x, Du^\epsilon) m^\epsilon dx\,dt 
&+\int_0^T\int_{\Tt^d} G(\eta_\epsilon*m^\epsilon) dxdt\\ 
\qquad
&\leq CT +\osc u^\epsilon(\cdot , T)+\int_0^T\int_{\Tt^d}
G(\eta_\epsilon*\mu) dxdt.    
\end{align*}
Since $\mu$ is bounded $G(\eta_\epsilon*\mu)$ is also bounded and 
so we obtain
\eqref{ihm}. 
\end{proof}

The estimates proved in this section are independent on $\epsilon$ and also 
do not rely on the positivity of $g$.
For instance they apply for $g(m)=\ln m$. 
Therefore
we improve slightly the results in   \cite{ll2}, whose proof depends
on lower bounds for $g$ (note however that in \cite{ll3} this lower
bound requirement is no longer asked 
though no proof is given there). In \cite{ll2} the
lower bound on $g$ is used to obtain a lower bound for $u$, 
which we managed to avoid here using a convexity argument.
\begin{Corollary}
\label{capu}
Assume A\ref{ah}-\ref{ag2} hold. 
Let $(u^\epsilon,m^\epsilon)$ be a solution of \eqref{eq:smfg}. Then 
\[
\int_0^T\int_{\Tt^d}(\eta_\epsilon*m^\epsilon)^{\alpha+1} 
+H(x,Du^\epsilon) m^\epsilon dxdt \leq C. 
\]
\end{Corollary}

\section{Further regularity for the Hamilton-Jacobi equation}
\label{best}

We now apply the results from the previous section to obtain improved
regularity for the Hamilton-Jacobi equation, the first equation in
\eqref{eq:smfg}. The estimates proved here depend in an essential  
way on $g$ being positive. The following proposition is an elementary
estimate on $u^\epsilon$ by below in the case $g\geq 0$,  
which will be crucial in what follows.  
\begin{Proposition}
\label{plest}
Let $(u^\epsilon,m^\epsilon)$ be a solution of \eqref{eq:smfg}. 
Suppose $g\geq 0$ and let 
$M=\max\limits_x H(x,0)$.
Then
\begin{equation}
\label{lest}
u^\epsilon(x,t) \geq  \min_xu^\epsilon(x,T)+M(t-T).
\end{equation}
\end{Proposition}
\begin{proof}
This is a simple application of the maximum principle. 
\end{proof}

The following estimate is also used in \cite{CLLP}.

\begin{Proposition}
\label{P41}
Assume A\ref{ah}-\ref{ag2} hold. 
Let $(u^\epsilon,m^\epsilon)$ be a solution of \eqref{eq:smfg}. 
We have
\begin{equation}
\label{hest}
\int_0^T\int_{\Tt^d}H(x,D_xu^\epsilon)dx dt\leq C
+\int_{\Tt^d}\left(u^\epsilon(x,T)-u^\epsilon(x,0)\right) dx.  
\end{equation}
\end{Proposition}
\begin{proof}
Integrating the first equation of \eqref{eq:smfg} and using Corollary
\ref{capu} we have 
\begin{align}
  \label{eq:qh1}
   \int_0^T\int_{\Tt^d}H(x, D_xu^\epsilon)dx\,dt
&=\int_0^T\int_{\Tt^d} (\eta_\epsilon*m^\epsilon)^\alpha dx dt+\int_{\Tt^d}\left(u^\epsilon(x, T)-u^\epsilon(x,0)\right)dx\\
&\notag \le C + \int_{\Tt^d}\left(u^\epsilon(x,T) -u^\epsilon(x,0)\right)dx.
\end{align}
\end{proof}

\begin{Remark}
From \eqref{hest} it follows, using $H\geq 0$, 
\begin{equation}
\label{ilb}
\int_{\Tt^d} u^\epsilon(x,0)dx \leq C+\int_{\Tt^d} u^\epsilon(x, T) dx. 
\end{equation}
\end{Remark}

\begin{Corollary}
\label{c4p1}
Assume A\ref{ah}-\ref{ag2} hold. 
Let $(u^\epsilon,m^\epsilon)$ be a solution of \eqref{eq:smfg}. 
Then
\[
\int_0^T\int_{\Tt^d} H(x, D_xu^\epsilon)dx\,dt \leq C+\osc(u^\epsilon(\cdot, T)),
\]
and
$
\int_{\Tt^d} |u^\epsilon(x,0)|dx\leq C+3\|u^\epsilon(\cdot,T)\|_{\infty}.
$
\end{Corollary}
\begin{proof}
It suffices to combine the lower bound from Proposition
\ref{plest}, with the estimate \eqref{hest}. 
\end{proof}


\section{Second order estimate}\label{elsoe}

We now discuss a second order estimate for mean-field games systems. This appeared in \cite{E2}, for the stationary case, as well as in \cite{LIMA}, concerning the time-dependent setting.


\begin{Proposition}\label{elsoe1}
Assume A\ref{ah}-\ref{bcc}
hold. 
Let $(u^\epsilon,m^\epsilon)$ be a solution of \eqref{eq:smfg}. Then
\begin{align*}
&\int_0^T\int_{\Tt^d}  g'(\eta_\epsilon*m^\epsilon)|D_x (\eta_\epsilon*m^\epsilon)|^2
+
\tr(D_{pp}^2H(D^2_{xx}u^\epsilon)^2) m^\epsilon
\leq \max_x\Delta u^\epsilon(x,T)\\&+
 C(1+\max_x u^\epsilon(x,T) -\min_xu^\epsilon(x, T))-\int_{\Tt^d} u^\epsilon(x,0)\Delta m^\epsilon(0,x) dx.
\end{align*}
\end{Proposition}
\begin{proof}
Applying the Laplacian $\Delta$ to first equation of \eqref{eq:smfg}
we get
\begin{align*}
& -\Delta u_t^\epsilon-
\Delta\Delta u^\epsilon+\tr(D_{pp}^2H(D^2_{xx}u^\epsilon)^2)
+\Delta_xH+2\tr(D_{px}^2HD^2_{xx}u^\epsilon)\\
&+D_pHD_x\Delta u^\epsilon =\diver(\eta_\epsilon*(g'(\eta_\epsilon*m^\epsilon)D_x (\eta_\epsilon*m^\epsilon)).
\end{align*}
Multiplying by $m^\epsilon$, integrating by parts and taking into account the
second equation of \eqref{eq:smfg}
and Assumption A\ref{strong}, we have
\begin{align}\label{eq:elsoe1}
&\int_0^T\int_{\Tt^d}g'(\eta_\epsilon*m^\epsilon)|D_x(\eta_\epsilon*m^\epsilon)|^2
+\tr(D_{pp}^2H(D^2_{xx}u^\epsilon)^2) m^\epsilon\\\notag
&\le \int_0^T\int_{\Tt^d}(|D_{xx}^2H|+\delta \tr(D_{pp}^2H(D^2_{xx}u^\epsilon)^2) + C_\delta H )m^\epsilon dx dt\\\notag
&\qquad +\int_{\Tt^d}m^\epsilon(x,T)\Delta u^\epsilon(x,T)-u(x,0)\Delta m^\epsilon(x,0) dx.
\end{align}
Choosing $\delta =\frac 1 2$ we obtain
\begin{align*}
&\int_0^T\int_{\Tt^d}g'(\eta_\epsilon*m^\epsilon)|D_x(\eta_\epsilon*m^\epsilon)|^2
	+\frac 1 2 \tr(D_{pp}^2H(D^2_{xx}u^\epsilon)^2)m^\epsilon dx dt\\
&\le C+C\int_0^T\int_{\Tt^d}Hm^\epsilon dx\,dt +\max_x\Delta
u(x,T)-\int_{\Tt^d}u^\epsilon(x,0)\Delta m^\epsilon(x,0) dx\\
&\le C(1+\max_x u^\epsilon(x,T) -\min_xu^\epsilon(x, T))+\max_x\Delta
u^\epsilon(x,T) -\int_{\Tt^d} u^\epsilon(x,0)\Delta m^\epsilon(0,x)dx,  
\end{align*}
where in the last inequality we used Proposition \ref{pehm}.
\end{proof}



Finally, from Sobolev's theorem we can conclude that 
\begin{Corollary}
\label{mhjr}
Assume  A\ref{ah}-\ref{bcc} hold. 
Let $(u^\epsilon,m^\epsilon)$ be a solution of \eqref{eq:smfg}. 
Then
\[
\int_0^T\int_{\Tt^d}g'(\eta_\epsilon*m^\epsilon)|D_x(\eta_\epsilon*m^\epsilon)|^2dx dt\leq C,  
\]
and so
\[
\int_0^T \|\eta_\epsilon*m^\epsilon\|_{L^{\frac{2^*}2 (\alpha+1)}(\Tt^d)}^{\alpha +1}dt\leq C.
\]
\end{Corollary}


\begin{Corollary}
\label{ctpt}
Assume A\ref{ah}-\ref{asbdpph} hold. 
Let $(u^\epsilon,m^\epsilon)$ be a solution of \eqref{eq:smfg}. Then
\[
\int_0^T \int_{\Tt^d} |\div D_pH|^2 m^\epsilon dxdt\leq C.
\]
\end{Corollary}
\begin{proof}
Observe that 
$\div(D_pH)=\tr (D^2_{pp} H D^2u) +\tr(D^2_{xp}H)$. 
Hence
$
|\div(D_pH)|^2\leq 
2 |\tr (D^2_{pp} H D^2u)|^2+2 |\tr (D^2_{xp}H)|\leq C \tr(D^2_{pp} H (D^2u)^2) + CH 
$,
by using Assumption A\ref{asbdpph}. The result then follows by applying a second order estimate, Proposition \ref{elsoe1}, and 
Proposition \ref{pehm}. 
\end{proof}

\section{Regularity for the Fokker-Planck equation}
\label{igain}

We first observe that by integrating the second equation of
\eqref{eq:smfg} we obtain $\int_{\Tt^d} m^\epsilon(x,
t)=1$, for all 
$0\leq t\leq T$.
Also the maximum principle yields that $m^\epsilon\geq 0$ if $m^\epsilon(x,0)\geq 0$.

In this section 
we explore various estimates and obtain further integrability for $m^\epsilon$. In Section \ref{rele}
we use the second order estimate
from the previous section to obtain improved integrability for $m$. 
 In Section \ref{elep} we control the integrability of $m^\epsilon$ in 
terms of $L^p$ norms of $D_pH$. For our purposes, we need explicit control for norms of $m^\epsilon$
in terms of polynomial expressions in $\|D_pH\|_{L^p(\Tt^d)}$.

\subsection{Regularity by the second order estimate}
\label{rele}

Our first approach concerning the regularity of the Fokker-Planck equation
uses the second order estimate. We start with an elementary result:

\begin{Proposition} \label{ent-prod}
Assume A\ref{ah} holds. 
Let $(u^\epsilon,m^\epsilon)$ be a solution of \eqref{eq:smfg}. 
Let $\varphi:\Rr\to \Rr$ be a $C^2$ function. Then
\[
\frac{d}{dt}\int_{\Tt^d} \varphi(m^\epsilon)+\int_{\Tt^d}\div(D_pH)
\varphi^*(m^\epsilon)=-\int_{\Tt^d} \varphi''(m^\epsilon) |D_x m^\epsilon|^2, 
\]
where $\varphi^*(z)=-z\varphi'(z)+\varphi(z)$. 
\end{Proposition}
\begin{proof} This is an elementary computation whose proof we will omit. 
\end{proof}

Next we obtain the following a priori estimates for $m^\epsilon$:
\begin{teo}
\label{em}
Assume A\ref{ah}-\ref{asbdpph} hold.
Let $(u^\epsilon,m^\epsilon)$ be a solution of \eqref{eq:smfg}. Then
for $d>2$, $\|m^\epsilon\|_{L^\infty([0,T],L^r(\Tt^d))}$ bounded for any 
$1\le r<\frac{2^*}{2}$, uniformly in $\epsilon$.  
\end{teo}
\begin{proof}
To simplify the notation, throughout this proof we will omit the $\epsilon$.
We will define inductively an increasing sequence $\be_n$ such 
that $\|m(\cdot, t)\|_{1+\be_n}$ is bounded. 
Set $\be_0=0$ so that $\|m(\cdot, t)\|_{1+\be_0}=1\leq C$.
 Let $\be_{n+1}=\dfrac 2d(\be_n+1)$,
then $\be_n$ is the $n^{th}$ partial sum of the geometric series with term
$\dfrac{2^n}{d^n}$. Thus 
$\lim_{n\to \infty}\be_n=\frac 2{d-2}=\frac{2^*}{2}-1$.
 
Set $q_n=
\frac{2^*}2 (\be_{n+1}+1)=
\frac{d}{d-2}(\be_{n+1}+1)$
Note that $q_n>\frac{d}{d-2}\be_{n+1}+\be_{n+1}+1>2\be_{n+1}+1$.
Then we have
\[
\|m\|_{2\be_{n+1}+1}\leq \|m\|_{1+\be_n}^{1-\lambda_n}\|m\|_{q_n}^{\lambda_n},
\]
for $0<\lambda_n<1$ defined through 
$\frac{\lambda_n}{q_n}+\frac{1-\lambda_n}{1+\be_n}=\frac{1}{2\be_{n+1}+1}$.
In particular
\begin{align}
  \label{eq:exponent}
\lambda_n&=\frac{q_n}{q_n-\be_n-1}\frac{2\be_{n+1}-\be_n}{1+2\be_{n+1}}=\frac{\be_{n+1}+1}{1+2 \be_{n+1}}.
\end{align}

Since $\| m \|_{1+\be_n}\leq C,$ we get
\begin{equation}
 \label{igm-eq5a}
\int_{\Tt^d} m^{2\be_{n+1} +1} dx
=\|m\|_{2\be_{n+1}+1}^{2\be_{n+1}+1}
\leq C\|m\|_{q_n}^{\lambda_n (2\be_{n+1}+1)}=
C\|m\|_{q_n}^{\be_{n+1} +1}.
\end{equation}
For $\be>0$, using Proposition \ref{ent-prod} with
$\varphi(m)=m^{\be+1}$ we obtain
\begin{align}
\label{igm-eq2} 
\int_{\Tt^d} m^{\be+1}(x,\tau)dx &+\frac{4\be}{\be+1}\int_0^\tau\int_{\Tt^d}
|D_xm^{\frac{\be+1}{2}}|^2dx\,dt\notag \\
&=  \int_{\Tt^d} m^{\be+1}(x,0)dx + \be
\int_0^\tau \int_{\Tt^d} \div(D_pH)m^{\be+1}dx dt.
\end{align}
Additionally we have
\begin{align}
\label{igm-eq4} 
 \int_{\Tt^d} |\div(D_pH)m^{\be+1}|dx
&\leq  C_\delta\left(\int_{\Tt^d} |\div(D_pH)|^2 m dx\right)
 + \delta\left(\int_{\Tt^d} m^{2\be +1} dx\right),
\end{align}
where all integrals are evaluated at a fixed time $t$. 

Setting $\be=\be_{n+1}$, from \eqref{igm-eq5a}, \eqref{igm-eq2} and
\eqref{igm-eq4} we get for any $\tau\in[0,T]$ 
\begin{align}
\label{igm-eq6a} 
&\int_{\Tt^d} m^{\be_{n+1}+1}(x,\tau)dx +\frac{4\be_{n+1}}{\be_{n+1}+1}
\int_0^\tau\int_{\Tt^d}|D_xm^{\frac{\be_{n+1}+1}{2}}(x,t)|^2dx\,dt
\notag \\
\leq  &\int_{\Tt^d} m^{\be_{n+1}+1}(x,0)dx + C_\delta \int_0^\tau\int_{\Tt^d}
|\div(D_pH)|^2 m dx\,dt +\delta \int_0^\tau
\|m\|_{q_n}^{\be_{n+1}+1} dt.
\end{align}

By Sobolev's theorem we get
\begin{equation}
 \label{igm-eq3a}
\|m\|_{q_n}^{\be_{n+1}+1}=\|m^{\frac{\be_{n+1}+1}{2}}\|_{2^*}^2
\le C\left(\int_{\Tt^d}m^{(\be_{n+1}+1)}(x,t)dx+
\int_{\Tt^d}|Dm^{\frac{\be_{n+1}+1}{2}}(x,t)|^2 dx\right). 
\end{equation}
From \eqref{igm-eq5a} and  $\int_{\Tt^d}m(x,t)dx=1$, for each fixed $t$
we have 
\begin{align*}
\int_{\Tt^d}m^{(\be_{n+1}+1)}dx&\le \left(\int_{\Tt^d}m^{(2\be_{n+1}+1)}dx\right)^{1/2}\le
C_\zeta+\zeta\int_{\Tt^d}m^{(2\be_{n+1}+1)}\\
&\le C_\zeta+\zeta \|m\|_{q_n}^{\be_{n+1}+1}.
\end{align*}
Thus
\begin{equation}
 \label{}
\|m\|_{q_n}^{\be_{n+1}+1}\le C \int_{\Tt^d}|Dm^{\frac{\be_{n+1}+1}{2}}(x,t)|^2 dx+C_\zeta+\zeta \|m\|_{q_n}^{\be_{n+1}+1}.
 \end{equation}

From \eqref{igm-eq6a} and \eqref{igm-eq3a}, taking $\de$ and
$\zeta$ small enough we have for some $\de_1>0$

\begin{align*}
\int_{\Tt^d} m^{\be_{n+1}+1}(x,\tau)dx &+\delta_1 \int_0^\tau \|m\|_{q_n}^{\be_{n+1}+1}dt\\
\leq  &C+C\int_{\Tt^d} m^{\be_{n+1}+1}(x,0)dx + C\int_0^\tau\int_{\Tt^d}
|\div(D_pH)|^2 m dx\,dt.
\end{align*}
Since the last term in the right-hand side is bounded by Corollary \ref{ctpt} we have established the result. 
\end{proof}

The proof of the previous Theorem also yields the following Corollary
\begin{Corollary}\label{DmaL2}
Assume A\ref{ah}-\ref{asbdpph} hold.
Let $(u^\epsilon,m^\epsilon)$ be a solution of \eqref{eq:smfg}. Then, for $-\frac 12\leq\be\le 0$ 
we have  
\begin{equation}
\int_0^T \int_{\Tt^d} (m^\epsilon)^{\be-1}|D_xm^\epsilon|^2 dxdt\leq C. 
\end{equation}
\end{Corollary}
\begin{proof}
To simplify the notation, as before, throughout this proof we will omit the $\epsilon$.
We first observe that for $-1\leq \be\leq 0$
we have
\[
\int_{\Tt^d} m^{\be+1} dx\leq C, 
\]
since for each fixed $t$ we have that $m(\cdot, t)$ is a probability
measure. Then, using   identity \eqref{igm-eq2}, coupled with estimate
\eqref{igm-eq4} and Corollary \ref{ctpt} yields 
\[
\int_0^\tau\int_{\Tt^d}
|D_xm^{\frac{\be+1}{2}}|^2dx\,dt\leq C+ C \int_0^\tau\int_{\Tt^d} m^{2 \be+1} dxdt, 
\]
and provided $-\frac 1 2 \leq \be\leq 0$ the right hand side is bounded. 
\end{proof}

As final remark, we would like to note that since the convolution with $\eta_\epsilon$ is 
a contraction in any $L^p$ space, we have
$m^\epsilon \in L^p$ implies $\eta_\epsilon*m^\epsilon \in L^p$. Similarly, 
if $(m^\epsilon)^\alpha \in L^p$ then we also have
$\eta_\epsilon*(\eta_\epsilon*m^\epsilon)^\alpha\in L^p$, and, of course, all these
bounds do not depend on $\epsilon$. 

\subsection{Regularity by $L^p$ estimates}
\label{elep}

We now obtain estimates for $m^\epsilon$ in $L^\infty([0,T],L^p(\Tt^d))$ depending polynomially on the $L^p$-norm of $D_pH$, for $p>\frac{d}{2}$. 
Because we need explicit estimates, we will prove them in detail. Throughout this Section, we omit the $\epsilon$ in the proofs for ease of presentation.

\begin{Lemma}\label{lem-ent-prod}
Let $(u^\epsilon,m^\epsilon)$ be a solution to \eqref{eq:smfg}. Then, for $\be>1$, there exist constants $c,\;C>0$ such that
\[\frac{d}{dt}\int_{\mathbb{T}^{d}}(m^\epsilon)^{\be}dx\leq
C\int_{\mathbb{T}^{d}}\left|D_{p}H\right|^{2}(m^\epsilon)^{\be}dx-c\int_{\mathbb{T}^{d}}\left|D_{x}(m^\epsilon)^{\frac{\be}{2}}\right|^{2}dx.\] 
\end{Lemma}
\begin{proof}
It follows from Proposition \ref{ent-prod} with
$\varphi(m)=m^\be$.
\end{proof}


We address now improved integrability of $m$ in terms of the $L^r(0,T;L^p(\Tt^d))$-norms of $|D_pH|^2$ for $p<\infty$.

\begin{Lemma}\label{lpp1}
Let $(u^\epsilon,m^\epsilon)$ be a solution of \eqref{eq:smfg} and assume that $\be\ge\be_0$ 
for $\be_0>1$ fixed.
\begin{equation}
\label{elpp1}
\frac{d}{dt}\int_{\Tt^d}(m^\epsilon)^{\be}(t,x)dx\leq C\left\|\left|D_{p}H\right|^{2}\right\|_{L^p(\Tt^d)}\left\|(m^{\epsilon})^{\be}\right\|_{L^q(\Tt^d)}
-c\int_{\Tt^d}\left|D_{x}((m^\epsilon)^{\frac \be 2})\right|^{2}dx,
\end{equation}
where 
$\frac{1}{q}+\frac{1}{p}=1$.
\end{Lemma}
\begin{proof}
Follows from
Lemma \ref{lem-ent-prod}.
\end{proof}

\begin{Definition}\label{def:def1}
Let $1\;\leq \;\be_{0}\;<\;\frac{2^{*}}{2}=\frac{d}{d-2}$ be a fixed constant. The
sequence $\left(\be_{n}\right)_{n\in\mathbb{N}}$ is defined inductively by 
$\be_{n+1}\;\doteq\;\theta\be_{n},$
where $\theta>1$ is a fixed constant.
\end{Definition}

\begin{Lemma}\label{lpp2} 
Assume that $(\be_n)_{n\in\mathbb{N}}$ is given as above and
let $1\;<\;q\;<\;\frac{d}{d-2}.$
Then 
\begin{align*}
\left\|\left(m^\epsilon\right)^{\be_{n+1}}\right\|_{q}\leq\left(\int_{\Tt^d}\left(m^\epsilon\right)^{\be_{n}}dx\right)^{\theta\kappa}
\left(\int_{\Tt^d}\left(m^\epsilon\right)^{\frac{2^*\be_{n+1}}{2}}\right)^{\frac{2\left(1-\kappa\right)}{2^{*}}},
\end{align*}
where $\kappa$ is given by \eqref{bbkappa}.
\end{Lemma}
\begin{proof}
H\"older inequality gives
\begin{align*}
\left(\int_{\Tt^d}m^{\be_{n+1}q}\right)^{\frac{1}{\be_{n+1}q}}\leq
\left(\int_{\Tt^d}m^{\be_{n}}\right)^{\frac{\kappa}{\be_{n}}}
\left(\int_{\Tt^d}m^{\frac{2^{*}\be_{n+1}}{2}}\right)^{\frac{2\left(1-\kappa\right)}{2^{*}\be_{n+1}}},
\end{align*}where
\begin{align}\label{eq:eq16}
\frac{1}{\be_{n+1}q}=\frac{\kappa}{\be_{n}}+\frac{2\left(1-\kappa\right)}{2^{*}\be_{n+1}}.
\end{align}
By rearranging the exponents one obtains 
\begin{align*}
\left(\int_{\Tt^d}m^{q\be_{n+1}}\right)^{\frac{1}{q}}\leq
\left(\int_{\Tt^d}m^{\be_{n}}\right)^{\frac{\be_{n+1}\kappa}{\be_{n}}}
\left(\int_{\Tt^d}m^{\frac{2^{*}\be_{n+1}}{2}}\right)^{\frac{2\left(1-\kappa\right)}{2^{*}}},
\end{align*}
establishing the result. The expression for $\kappa$ follows from 
(\ref{eq:eq16}) taking into account the definition of $(\be_{n})_{n\in\mathbb{N}}$. 
The condition on $q$ follows from the requirement that $\kappa>0$.
\end{proof}

\begin{Lemma}\label{lpp3}
For any $1<q<\frac{d}{d-2}$ we have
\begin{align*}
\left\|(m^{\epsilon})^{\frac{\be_{n+1}}{2}}\right\|_{2^{*}}^{2\left(1-\kappa\right)}\leq  C\left(\int_{\Tt^d}
\left|D_{x}\left((m^{\epsilon})^{\frac{\be_{n+1}}{2}}\right)\right|^{2}\right)^{\left(1-\kappa\right)}
+C\left\|(m^{\epsilon})^{\be_{n+1}}\right\|_{q}^{\left(1-\kappa\right)}.
\end{align*}
\end{Lemma}
\begin{proof}As before, we drop the superscript $\epsilon$.
From Sobolev inequality
\begin{align}\label{eq:eq06}
\left\|m^{\frac{\be_{n+1}}{2}}\right\|_{2^{*}}^{2\left(1-\kappa\right)}
\leq C\left(\int_{\Tt^d}\left|D_{x}\left(m^{\frac{\be_{n+1}}{2}}\right)\right|^{2}\right)^{\left(1-\kappa\right)}
+C\left(\int_{\Tt^d}\left|m^{\be_{n+1}}\right|\right)^{\left(1-\kappa\right)}.
\end{align}
Applying H\"older's inequality to the last term of this inequality yields 
\begin{align}\label{eq:eq21}
\left(\int_{\Tt^d}\left|m^{\be_{n+1}}\right|\right)^{\left(1-\kappa\right)}\leq C\left\|m^{\be_{n+1}}\right\|_{q}^{\left(1-\kappa\right)}.
\end{align}
Combining \eqref{eq:eq21} with \eqref{eq:eq06} leads to the result.
\end{proof}

\begin{Lemma}\label{lpp4}
Assume that $1\;<\;q\;<\;\frac{d}{d-2}$. Then
$\left\|(m^{\epsilon})^{\be_{n+1}}\right\|_{q}\leq C+\delta\left\|(m^{\epsilon})^{\frac{\be_{n+1}}{2}}\right\|_{2^{*}}^{2}$.
\end{Lemma}
\begin{proof}
Since $1<q \be_{n+1}<\frac{2^{*}\be_{n+1}}{2}$, one can interpolate between these exponents to obtain
\begin{align*}
\left(\int_{\Tt^d}m^{\be_{n+1}q}\right)^{\frac{1}{\be_{n+1}q}}
\leq\left(\int_{\Tt^d}m\right)^{\lambda}\left(\int_{\Tt^d}m^{\frac{2^{*}\be_{n+1}}{2}}\right)^{\frac{2\left(1-\lambda\right)}{2^{*}\be_{n+1}}},
\end{align*}
where
$0<\lambda<1$ solves
$\frac{1}{\be_{n+1}q}=\lambda+\frac{2\left(1-\lambda\right)}{2^{*}\be_{n+1}}$.
Because $m$ is a probability measure, it follows
$\left\|m^{\be_{n+1}}\right\|_{q}\leq\left\|m^{\frac{\be_{n+1}}{2}}\right\|_{2^{*}}^{2\left(1-\lambda\right)}$. 
Furthermore, since $\left(1-\lambda\right)<1$, Young's inequality with
$\delta$ leads to 
$\left\|m^{\be_{n+1}}\right\|_{q}\leq
C+\delta\left\|m^{\frac{\be_{n+1}}{2}}\right\|_{2^{*}}^{2}$, 
establishing the result.
\end{proof}

\begin{Proposition}\label{ppp1}
Assume that $1\;<\;q\;<\;\frac{d}{d-2}.$ Then
\begin{align*}
\left\|(m^{\epsilon})^{\be_{n+1}}\right\|_{q}\leq\left(\int_{\Tt^d}(m^{\epsilon})^{\be_{n}}\right)^{\theta\kappa}
\left[C+C\left(\int_{\Tt^d}\left|D_{x}\left((m^{\epsilon})^{\frac{\be_{n+1}}{2}}\right)\right|^{2}\right)^{\left(1-\kappa\right)}\right].
\end{align*}
\end{Proposition}
\begin{proof}
By combining the statement of Lemmas \ref{lpp3} and \ref{lpp4} one obtains that
\begin{align*}
\left\|m^{\frac{\be_{n+1}}{2}}\right\|_{2^{*}}^{2\left(1-\kappa\right)}\leq
C+C\left(\int_{\Tt^d}\left|D_{x}\left(m^{\frac{\be_{n+1}}{2}}\right)\right|^{2}\right)^{\left(1-\kappa\right)}
+\delta\left\|m^{\frac{\be_{n+1}}{2}}\right\|_{2^{*}}^{2\left(1-\kappa\right)}.
\end{align*}
Absorbing the last term of this inequality in the left-hand side, it follows that 
\begin{align}\label{eq:eq17}
\left\|m^{\frac{\be_{n+1}}{2}}\right\|_{2^{*}}^{2\left(1-\kappa\right)}
\leq C+ C\left(\int_{\Tt^d}\left|D_{x}\left(m^{\frac{\be_{n+1}}{2}}\right)\right|^{2}\right)^{\left(1-\kappa\right)}.
\end{align}
By using \eqref{eq:eq17} in Lemma \ref{lpp2} one obtains the result.
\end{proof}

\begin{Proposition}\label{ppp2}
Let $(u^\epsilon,m^\epsilon)$ be a solution of \eqref{eq:smfg} and assume that 
$1\;<\;q\;<\;\frac{d}{d-2}$.
Let $\frac{1}{q}+\frac{1}{p}=1$
and $r\kappa=1$ where $\kappa$ is given by \eqref{bbkappa}. Then, 
\begin{align}\label{eq:eq18}
\frac{d}{dt}\int_{\Tt^d}(m^{\epsilon})^{\be_{n+1}}dx\leq
&C+ C\left\|\left|D_{p}H\right|^{2}\right\|_{L^p(\Tt^d)}^{r}\left(\int_{\Tt^d}\left(m^\epsilon\right)^{\be_{n}}dx\right)^{\theta}.
\end{align}
\end{Proposition}
\begin{proof}
From \eqref{elpp1} in 
Lemma \ref{lpp1}, 
using Proposition \ref{ppp1}, one obtains that
\begin{align*}
&\frac{d}{dt}\int\limits_{\Tt^d}m^{\be_{n+1}}(t,x)dx
\leq \left\|\left|D_{p}H\right|^{2}\right\|_{p}\left(\int_{\Tt^d}m^{\be_{n}}\right)^{\theta\kappa}
\left[C\left(\int_{\Tt^d}\left|D_{x}\left(m^{\frac{\be_{n+1}}{2}}\right)\right|^{2}\right)^{\left(1-\kappa\right)}+C\right]\\ 
&-c\int_{\Tt^d}\left|D_{x}\left(m^{\frac{\be_{n+1}}{2}}\right)\right|^{2}dx
\leq C\left\|\left|D_{p}H\right|^{2}\right\|_{p}\left(\int_{\Tt^d}m^{\be_{n}}\right)^{\theta\kappa}
+C\left\|\left|D_{p}H\right|^{2}\right\|_{p}^{r}\left(\int_{\Tt^d}m^{\be_{n}}\right)^{\theta}\\
&\leq C+ C\left\|\left|D_{p}H\right|^{2}\right\|_{p}^{r}\left(\int_{\Tt^d}m^{\be_{n}}\right)^{\theta}, 
\end{align*}
where the last inequality follows from Young's inequality and using
the fact that $r>1$ and $r\kappa=1$. This concludes the proof.
\end{proof}

\begin{proof}[Proof of the Theorem \ref{cpp2}.]
The proof proceeds by induction on $n$.
For $n=1$ we integrate \eqref{eq:eq18} with respect to $dt$ over $(0,\tau)$ to obtain
\begin{align*}
\int_{\Tt^d}m^{\be_{1}}\left(\tau,x\right)dx&\leq C\int_{0}^{\tau}\left\|\left|D_{p}H\right|^{2}\right\|_{L^p(\Tt^d)}^{r}dt+C
\leq C\left\|\left|D_{p}H\right|^{2}\right\|_{L^r(0,T;L^p(\Tt^d))}^{r}+C,
\end{align*}
where we used that 
$\int_{\Tt^d}m^{\be_0}dx\leq C$
for some constant $C>0$. This verifies our claim for $n=1$.
%
%
%

Then,
\begin{align}\label{eq:der}
\frac{d}{dt}\int_{\mathbb{T}^d}m^{\be_{n+1}}dx&
\leq C\left\|\left|D_pH\right|^2\right\|^r_{L^p(\Tt^d)}
\left(C+C\left\|\left|D_{p}H\right|^{2}\right\|_{L^{r}\left(0,T;L^{p}\left(\Tt^d\right)\right)}^{r_{n}}\right)^{\theta}. 
\end{align}
Integrating \eqref{eq:der} with respect to the Lebesgue measure $dt$  over $(0,\tau)$
one obtains that 
\begin{align*}
\int_{\Tt^d}m^{\be_{n+1}}\left(\tau,x\right)dx
&\leq C\int_{0}^{\tau}\left\|\left|D_pH\right|^2\right\|_{L^p(\Tt^d)}^{r}\left\|\left|D_pH\right|^2\right\|^{r_{n}\theta}_{L^{r}\left(0,T;L^p(\Tt^d)\right)}dt\\
&\quad+ C\int_{0}^{\tau}\left\|\left|D_pH\right|^2\right\|_{L^p(\Tt^d)}^{r}dt+C.
\end{align*}A further application of H\"older inequality leads to 
\begin{align*}
\int_{\Tt^d}m^{\be_{n+1}}\left(\tau,x\right)dx&\leq C+C\left\|\left|D_pH\right|^2\right\|^{r+r_{n}\theta}_{L^{r}\left(0,T;L^p(\Tt^d)\right)},
\end{align*}
which establishes the result.	
\end{proof}

\section{Upper bounds for the Hamilton-Jacobi equation}
\label{pub}

In this section we investigate $L^\infty$ bounds for solutions to 
 the Hamilton-Jacobi equation. 
Since by Proposition \ref{plest} any solution to  \eqref{eq:smfg}
is bounded by below, to get these bounds it is enough to establish upper bounds. 
These build
upon the improved integrability obtained previously for $m^\epsilon$ and 
will be used in the following sections. As before, we omit the $\epsilon$ in the proofs in this Section.



\begin{Proposition}
\label{imphc0}
Suppose $(u^\epsilon,m^\epsilon)$ is a solution of \eqref{eq:smfg} and $H$ satisfies
A\ref{ah}. Then, if 
$p>\frac{d}{2}$, we have
\[
u^\epsilon(x,\tau)\le(T-\tau)\max_z L(z,0)+
C\|g_\ep(m)\|_{L^\infty(0,T;L^{p}(\Tt^d))}
+\int_{\Tt^d}u^\epsilon(y,T)\te(y,T-\tau)dy,  
\]
where $\te$ is the heat kernel, with
$\te(\cdot,\tau)=\delta_x$. Furthermore, if 
$\frac{1}{r}+\frac{1}{s}=\frac{1}{p}+\frac{1}{q}=1$, 
and $\frac{p}{s}>\frac{d}{2}$, we have 
\[
u^\epsilon(x,\tau)\le(T-\tau)\max_z L(z,0)+
C\|g_\ep(m)\|_{L^r(0,T;L^{p}(\Tt^d))}
+\int_{\Tt^d}u^\epsilon(y,T)\te(y,T-\tau)dy.
\]
\end{Proposition}
\begin{proof}
 By applying Proposition \ref{plb} with $b=0$ and $\zeta_0=\te(\cdot,\tau)=\delta_{x}$,  
we obtain the estimate
\begin{align*}
u(x,\tau)\le &(T-\tau)\max_{z\in \Tt^d} L(z,0)\\
&+
\int_\tau^T\int_{\Tt^d}g(m)(y,t)\te(y,t-\tau)dydt
+\int_{\Tt^d} u(y,T)\te(y,T-\tau)dy.
\end{align*}

It is clear that the key point is to estimate 
$
\int_\tau^T\int_{\Tt^d} g(m)(y,t)\te(y,t-\tau)dydt. 
$
We recall the following property of the heat kernel, 
for $\dfrac 1 p+\dfrac 1 q=1$ we have
$
\|\te(\cdot, t)\|_q\leq \frac{C}{t^{\frac{d}{2p}}}.
$
Hence,
\[
\int_{\Tt^d} g(m)(y,t)\te(y,t-\tau)dy
\leq \frac{C}{(t-\tau)^{\frac{d}{2p}}}\|g(m(\cdot,t))\|_{L^p(\Tt^d)}. 
\]
Thus if $d<2p$ we have
\[
\int_\tau^T\int_{\Tt^d}g(m)(y,t)\te(y,t-\tau)dydt\le 
C\|g(m)\|_{L^\infty(0,T;L^{p}(\Tt^d))}.
\]

For the second assertion,
H\"older inequality leads to 
\begin{align*}
\int_{\tau}^{T}\int_{\Tt^d}&g(m)(x,t)\te(x,t-\tau)dxdt
\leq\int_{\tau}^{T}\|g(m)(\cdot,t)\|_{L^{p}(\Tt^d)}\|\te(\cdot,t-\tau)\|_{L^{q}(\Tt^d)}dt\\
&\leq\|g(m)\|_{L^{r}(0,T;L^{p}(\Tt^d))}\left(\int_{\tau}^{T}\frac{C}{t^{\frac{ds}{2p}}}\right)^{\frac{1}{s}}
\leq C\|g(m)\|_{L^{r}(0,T;L^{p}(\Tt^d))},
\end{align*}
where the last inequality follows from $\frac{ds}{2p}<1$.
\end{proof}

\begin{Corollary}
\label{imphc}
Suppose  A\ref{ah}-\ref{bcc}  hold.
Let $(u^\epsilon,m^\epsilon)$ be a solution of \eqref{eq:smfg}.  
Then, for any  $p>\frac{d}{2}$ we have 
\[
u^\epsilon(x,\tau)\le(T-\tau)\max_z L(z,0)+
C\|m*\eta_\ep\|^\alpha_{L^\infty(0,T;L^{\alpha p}(\Tt^d))}
+\int_{\Tt^d}u^\epsilon(y,T)\te(y,T-\tau)dy,  
\]
where $\te$ is the heat kernel, with $\te(\cdot,\tau)=\delta_x$.
Furthermore, if $\alpha p\leq 1$ we have
\[
u^\epsilon(x,\tau)\le(T-\tau)\max_z L(z,0)+
C
+\int_{\Tt^d}u^\epsilon(y,T)\te(y,T-\tau)dy.
\] 
\end{Corollary}

\begin{Corollary}
\label{imphc2}
Suppose  A\ref{ah}-\ref{bcc}  hold.
Let $(u^\epsilon,m^\epsilon)$ be a solution of \eqref{eq:smfg}.  
Then, for any  $p,\;r$ such that $\frac{p(r-1)}{r}>\frac{d}{2}$, we have 
\[
u^\epsilon(x,\tau)\le(T-\tau)\max_z L(z,0)+
C\|m*\eta_\ep\|^\alpha_{L^{\alpha r}(0,T;L^{\alpha p}(\Tt^d))}
+\int_{\Tt^d}u^\epsilon(y,T)\te(y,T-\tau)dy.
\]
\end{Corollary}

We end this Section with the proof of Lemma \ref{l8421}.

\begin{proof}[Proof of the Lemma \ref{l8421}]
The result easily follows from the second assertion of Proposition \ref{imphc0}.
\end{proof}

\section{Sobolev regularity for the Hamilton-Jacobi equation}\label{rsbqc}

In this section we consider regularity in Sobolev spaces for the Hamilton-Jacobi equation. To do so, we start by recalling the Gagliardo-Nirenberg interpolation inequality.

\begin{Lemma}\label{l91}
Let $u\in W^{2,p}(\Tt^d)$. Then there exists a constant $C>0$ such that,
\begin{equation}\label{eq:e91}
\|Du\|_{L^{2p}(\Tt^d)}\leq C\|D^2u\|_{L^{p}(\Tt^d)}^{\frac{1}{2}}\|u\|_{L^{\infty}(\Tt^d)}^{\frac{1}{2}}.
\end{equation}
\end{Lemma}
\begin{proof}
Gagliardo-Nirenberg (see \cite{Fr}) inequality implies 
\[\|Du\|_{L^{2p}(\Tt^d)}\leq
C\|D^2u\|_{L^{r}(\Tt^d)}^{\frac{1}{2}}\|u\|_{L^{\infty}(\Tt^d)}^{\frac{1}{2}},\]
provided $
\frac{1}{2p}=\frac{1}{d}+\frac{1}{2}\left(\frac{1}{r}-\frac{2}{d}\right)
$. This identity yields $r=p$, which concludes the proof.
\end{proof}

\begin{Lemma}\label{l92}
Let $u\in W^{1,2p}(\Tt^d)$. Then, there exists $C>0$ such that 
\[\|Du\|_{L^{\ga p}(\Tt^d)}\leq C\|Du\|_{L^{2p}(\Tt^d)},\]for every $1<\ga<2$.
\end{Lemma}
\begin{proof}
This follows from H\"older's inequality.
\end{proof}

\begin{Corollary}\label{c91}
Let $u\in W^{2,p}(\Tt^d)$. Then, there exists $C>0$ such that 
\begin{equation}\label{eq:e91a}
\|Du\|_{L^{\ga p}(\Tt^d)}\leq C\|D^{2}u\|_{L^p(\Tt^d)}^{\frac 12}\|u\|_{L^\infty(\Tt^d)}^{\frac 12}.
\end{equation}
\end{Corollary}
\begin{proof}
The result follows by combining Lemmas \ref{l91} and \ref{l92}.
\end{proof}

\begin{Lemma}\label{lrsqc2}
Let $(u^\ep,m^\ep)$ be a solution of \eqref{eq:smfg}. Then, 
$$\|u^\ep_{t}\|_{L^{r}(0,T;L^{p}(\Tt^d))}, \|D^{2}u^\ep\|_{L^{r}(0,T;L^{p}(\Tt^d))} \le \|g_\ep(m^\ep)\|_{L^{r}(0,T;L^{p}(\Tt^d))}+\|H\|_{L^r(0,T;L^p(\Tt^d))},$$
for $1<p, r<\infty$. Furthermore,
\[
\|Du^\ep\|_{L^\infty([0,T],L^2(\Tt^d))}\le \|g_\ep(m^\ep)\|_{L^{2}(0,T;L^{2}(\Tt^d))}+\|H\|_{L^{2}(0,T;L^{2}(\Tt^d))}.
\]
\end{Lemma}
\begin{proof}
It follows from standard regularity theory for the heat equation. See \cite{Lad, LiFM}.
\end{proof}

\begin{Lemma}\label{lrsqc3}
Let $(u^\ep,m^\ep)$ be a solution of \eqref{eq:smfg} and assume that
A\ref{ah}-\ref{asbdpph} hold. For $1<p, r<\infty$ there are constants
$c, C > 0$ such that 
\[\|H(x,Du^\ep)\|_{L^{r}(0,T;L^{p}(\Tt^d))} \leq\;c\|D^{2}u^\ep\|_{L^{r}(0,T;L^{p}(\Tt^d))}^{\frac{\ga}{2} }\|u^\ep\|_{L^{\infty}(0,T;L^{\infty}(\Tt^d))}^{\frac{\ga}{2}}+C.\]
\end{Lemma}
\begin{proof}
For ease of presentation, we omit the $\epsilon$. Assumption A\ref{coerca} implies that 
$
\left(\int_{\Tt^d}|H(x,Du(x,t))|^{p}dx\right)^{\frac{1}{p}}\leq c\left(\int_{\Tt^d}|Du|^{\ga p}dx\right)^{\frac{1}{p}}+C.
$
By combining this with
Corollary \ref{c91} it follows that 
\begin{align*}
\left(\int_{\Tt^d}|H(x,Du(x,t))|^{p}dx\right)^{\frac{1}{p}}&\leq c\left(\int_{\Tt^d}|Du|^{\ga p}dx\right)^{\frac{1}{p}}+C\\
&\leq c\|D^{2}u\|_{L^{p}(\Tt^d)}^{\frac{\ga}{2}}\|u\|_{L^{\infty}(\Tt^d)}^{\frac{\ga}{2}}+C.
\end{align*}
Then,
\begin{align*}
\|H(x,Du)\|_{L^r(0,T;L^p(\Tt^d))}&\le c\left(\int_{0}^{T}\Big(\|D^{2}u\|_{L^p(\Tt^d)}^{\frac\ga2}\|u\|_{L^{\infty}(\Tt^d)}^{\frac\ga2}\Big)^r\right)^{\frac 1r}+C\\
&\leq c\|u\|_{L^{\infty}(0,T;L^{\infty}(\Tt^d))}^{\frac\ga2}\|D^{2}u\|_{L^r(0,T;L^p(\Tt^d))}^{\frac\ga2}+C,
\end{align*}
where in the last inequality we used that $\frac \gamma 2 <1$. 
\end{proof}

The proof of Theorem \ref{c92} ends the section.

\begin{proof}[Proof of the Theorem \ref{c92}]
As before, we omit the $\epsilon$. By combining Lemmas \ref{lrsqc2} and \ref{lrsqc3} one obtains 
\begin{align*}
\|D^2u\|_{L^r(0,T;L^p(\Tt^d))} &\le c\|D^2u\|_{L^r(0,T;L^p(\Tt^d))}^{\frac\ga2}\|u\|_{L^{\infty}(0,T;L^{\infty}(\Tt^d))}^{\frac\ga2}\\
&+\|g(m)\|_{L^r(0,T;L^p(\Tt^d))}+C.
\end{align*}
Set $j=\frac{2}{\ga}$ and define $l$ by 
$\frac{1}{j}+\frac{1}{l}=1$.
Using Young's inequality with $\varepsilon$, with exponents $j$ and $l$, it follows that 
\begin{align*}
\|D^2u\|_{L^r(0,T;L^p(\Tt^d))}&\le \|g(m)\|_{L^r(0,T;L^p(\Tt^d))}+\varepsilon\|D^2u\|_{L
^r(0,T;L^p(\Tt^d))}\\
&+C\|u\|_{L^{\infty}(0,T;L^{\infty}(\Tt^d))}^{\frac\ga{2-\ga}}+C.
\end{align*}
Absorbing the term $\varepsilon\|D^{2}u\|_{L^r(0,T;L^p(\Tt^d))}$ on the left-hand side yields
\begin{align*}
\|D^2u\|_{L^r(0,T;L^{p}(\Tt^d))}\le c\|g(m)\|_{L^r(0,T;L^p(\Tt^d))}+c\|u\|_{L^{\infty}(0,T;L^{\infty}(\Tt^d))}^{\frac\ga{2-\ga}}+C,
\end{align*}
which concludes the proof.
\end{proof}

\section{Improved regularity}
\label{sec10}
Throughout this section we define, for $1\le\be_0<\dfrac{2^*}2$ and
\begin{equation}\label{eq:1sec10}
0\leq\upsilon\leq 1<\te,
\end{equation}
\begin{equation}\label{eq:2sec10}
b_\upsilon\doteq\frac{d(\af+1)\be_0\theta}{(\af+1)d\upsilon+\theta\be_0(d-2)(1-\upsilon)}\;\;\;\;\mbox{and}\;\;\;\;a_\upsilon\doteq\frac{\af+1}{1-\upsilon}.
\end{equation}

\begin{Lemma}\label{lemma2sec10}
Let $(u^\epsilon,m^\epsilon)$ be a solution of \eqref{eq:smfg} and assume that
A\ref{ah}-\ref{asbdpph} hold. 
Suppose further that $a_\upsilon$ and $b_\upsilon$ are given as in
\eqref{eq:1sec10} and \eqref{eq:2sec10} respectively. Then,  
\[
\big\|m^\epsilon\big\|_{L^{a_\upsilon}(0,T;L^{b_\upsilon}(\Tt^d))}\leq C+C\big\||D_pH|^2\big\|^\frac{r\upsilon\left(1-\frac{1}{\theta}\right)}{\be_0\left(\theta-1\right)}_{L^r(0,T;L^p(\Tt^d))},
\]
where 
\begin{equation}
\label{eq:c4}
p>\frac{d}{2}\;\;\,\;\mbox{and}\;\;\;\;r=\frac{p(d(\theta-1)+2)}{2p-d}.
\end{equation}
\end{Lemma}
\begin{proof}
For ease of presentation, we omit the $\epsilon$. H\"older inequality implies that 
\[
\big\|m\big\|_{L^{a_\upsilon}(0,T;L^{b_\upsilon}(\Tt^d))}\leq\big\|m\big\|_{L^{\af+1}(0,T;L^{\frac{2^*(\af+1)}{2}}(\Tt^d))}^{1-\upsilon}
\big\|m\big\|_{L^{\infty}(0,T;L^{\theta\beta_0}(\Tt^d))}^\upsilon
\]
provided that, for some $0\leq \upsilon\leq 1$,  
$
\frac{1}{a_\upsilon}=\frac{1-\upsilon}{\af+1}
$
and
$
\frac{1}{b_\upsilon}=\frac{1-\upsilon}{\frac{2^*(\af+1)}{2}}+\frac{\upsilon}{\theta\be_0},
$
which hold by \eqref{eq:1sec10} and \eqref{eq:2sec10}. 
Because of Corollary \ref{mhjr} we have 
$
\|m\|_{L^{\af+1}(0,T;L^{\frac{2^*(\af+1)}{2}}(\Tt^d))}^{1-\upsilon}\leq C.
$
On the other hand, from Theorem \ref{cpp2} it follows that 
\[
\big\|m\big\|_{L^{\infty}(0,T;L^{\theta\beta_0}(\Tt^d))}^\upsilon\leq C+C\big\||D_pH|^2\big\|_{L^r(0,T;L^p(\Tt^d))}^\frac{r\upsilon\left(1-\frac{1}{\theta}\right)}{\left(\theta-1\right)\be_0}.
\]
By combining the previous computations one obtains the result.
\end{proof}

\begin{Lemma}\label{lemma1sec10}
Let $(u^\epsilon,m^\epsilon)$ be a solution of \eqref{eq:smfg} and assume that
A\ref{ah}-\ref{asbdpph} hold.  
Assume that 
\begin{equation}\label{eq:3sec10}
\frac{b_\upsilon}{a_\upsilon}\big(\frac{a_\upsilon-\af}{\af}\big)>\frac{d}{2}.
\end{equation}
Then, 
\[
\big\|u^\ep\big\|_{L^\infty(0,T;L^\infty(\Tt^d))}\leq C+C\big\|g_\ep(m)\big\|_{L^\frac{a_\upsilon}{\af}(0,T;L^\frac{b_\upsilon}{\af}(\Tt^d))}
\]
\end{Lemma}
\begin{proof}
The result follows by using Lemma \ref{imphc0} since \eqref{eq:3sec10} holds.
\end{proof}

\begin{Lemma}\label{lem:newint}
Let $(u^\epsilon,m^\epsilon)$ be a solution of \eqref{eq:smfg} and assume that
A\ref{ah}-\ref{asbdpph} hold. Let $\zeta$, $\tilde{p}$ and $\tilde{r}$ be such that
\begin{equation}\label{eq:eq1ni}
0\leq \zeta\leq 1,\;\;\tilde{p}\left(\frac{\tilde{r}-1}{\tilde{r}}\right)>\frac{d}{2},
\end{equation}where 
\begin{equation}\label{eq:eq2ni}
\frac{1}{\tilde{p}}\doteq\frac{1-\zeta}{\left(1+\frac{1}{\af}\right)\frac{d}{d-2}}+\frac{\zeta}{\frac{b_\upsilon}{\af}}
\end{equation}and
\begin{equation}\label{eq:eq3ni}
\frac{1}{\tilde{r}}\doteq\frac{1-\zeta}{1+\frac{1}{\af}}+\frac{\zeta}{\frac{a_\upsilon}{\af}}.
\end{equation} Then
$$\left\|g_\ep\right\|_{L^{\tilde{r}}\left(0,T;L^{\tilde{p}}(\Tt^d)\right)}\leq C\left\|g_\ep\right\|_{L^{\frac{a_{\upsilon}}{\af}}\left(0,T;L^{\frac{b_\upsilon}{\af}}(\Tt^d)\right)}^{\zeta}$$and
$$\left\|u^\ep\right\|_{L^\infty(0,T;L^\infty(\Tt^d))}\leq C+C\left\|g_\ep(m)\right\|_{L^{\tilde{r}}(0,T;L^{\tilde{p}}(\Tt^d))}.$$
\end{Lemma}
\begin{proof}
The second assertion follows from \eqref{eq:eq1ni} along with Lemma \ref{l8421}. For the first assertion, notice that H\"older inequality implies
$$\left\|g_\ep\right\|_{L^{\tilde{r}}\left(0,T;L^{\tilde{p}}(\Tt^d)\right)}\leq\left\|g_\ep\right\|_{L^{1+\frac{1}{\af}}\left(0,T;L^{\frac{2^*}{2}\left(1+\frac{1}{\af}\right)}(\Tt^d)\right)}^{1-\zeta} C\left\|g_\ep\right\|_{L^{\frac{a_{\upsilon}}{\af}}\left(0,T;L^{\frac{b_\upsilon}{\af}}(\Tt^d)\right)}^{\zeta}.$$Also, we have from Corollary \ref{mhjr} that $\left\|g_\ep\right\|_{L^{1+\frac{1}{\af}}\left(0,T;L^{\frac{2^*}{2}\left(1+\frac{1}{\af}\right)}(\Tt^d)\right)}^{1-\zeta}<C$, for some $C>0$. By combining these, the result follows.
\end{proof}

\begin{Lemma}\label{lemma03sec10}
Let $(u^\epsilon,m^\ep)$ be a solution of \eqref{eq:smfg} and assume that
A\ref{ah}-\ref{asbdpph} hold. Suppose further that $p>\frac{d}{2}$ and $r$ is given as is \eqref{eq:c4}. Then 
\[
\left\|\left|D_pH\right|^2\right\|_{L^r(0,T;L^p(\Tt^d))}\leq C+C\left\|Du^\ep\right\|_{L^F(0,T;L^G(\Tt^d))}^{2(1-\lambda)(\gamma-1)},
\]where $\lambda$, $F$ and $G$ satisfy 
\begin{equation}\label{eq:eqF}
0\leq \lam\leq1,\;\;\;\frac{1}{2(\ga-1)r}=\frac{\lambda}{\gamma}+\frac{1-\lambda}{F}
\end{equation}
and
\begin{equation}\label{eq:eqG}
\frac{1}{2(\ga-1)p}=\frac{\lambda}{\gamma}+\frac{1-\lambda}{G},
\end{equation}respectively.
\end{Lemma}
\begin{proof}
Because of A\ref{dphs}, we have that 
\[
\left\|\left|D_pH\right|^2\right\|_{L^r(0,T;L^p(\Tt^d))}\leq C+C\left\|Du^\ep\right\|^{2(\ga-1)}_{L^{2(\ga-1)r}(0,T;L^{2(\ga-1)p}(\Tt^d))}.
\]
On the other hand, H\"older inequality implies that
\begin{equation}\label{eq:eq88plus}
\left\|Du^\ep\right\|_{L^{2(\ga-1)r}(0,T;L^{2(\ga-1)p}(\Tt^d))}\leq \left\|Du^\ep\right\|_{L^\ga(0,T;L^\ga(\Tt^d))}^\lambda\left\|Du^\ep\right\|_{L^F(0,T;L^G(\Tt^d))}^{1-\lambda}
\end{equation}
provided that \eqref{eq:eqF} and \eqref{eq:eqG} hold.
Because of Proposition \ref{pehm}, we have $Du\in L^\gamma(\Tt^d\times[0,T])$. By combining this with the previous computation, the result follows.
\end{proof}

\begin{Lemma}\label{lemma3sec10}
Let $(u^\epsilon,m^\ep)$ be a solution of \eqref{eq:smfg} and assume that
A\ref{ah}-\ref{asbdpph} hold. Suppose further that \eqref{eq:3sec10}-\eqref{eq:eqG},
\begin{equation}\label{eq:5sec10}
\frac{F}{\ga}=\frac{a_\upsilon}{\af}
\end{equation}
and
\begin{equation}\label{eq:6sec10}
\frac{G}{\ga}=\frac{b_\upsilon}{\af}
\end{equation}
hold. Then,
\[
\big\|Du^\ep\big\|_{L^F(0,T;L^G(\Tt^d))}\leq C\big\|g_\ep\big\|_{L^{\frac{a_\upsilon}{\af}}(0,T;L^{\frac{b_\upsilon}{\af}}(\Tt^d))}^{\frac{\zeta}{2-\ga}+\frac{\zeta}{2}}+C\big\|g_\ep\big\|_{L^{\frac{a_\upsilon}{\af}}(0,T;L^{\frac{b_\upsilon}{\af}}(\Tt^d))}^{\frac{1}{\ga}+\frac{\zeta}{2}}+C.
\]
\end{Lemma}
\begin{proof}
Inequality \eqref{eq:e91a} implies that 
\[
\|Du^\ep\|_{L^{2(\ga-1)p}(\Tt^d)}\leq C\|D^2u^\ep\|^{\frac{1}{2}}_{L^{\frac{2(\ga-1)p}{\ga}}(\Tt^d)}\|u^\ep\|_{L^\infty(\Tt^d)}^{\frac{1}{2}}.
\]
By noticing that $\ga <2$ it follows that
\begin{equation}
  \label{eq:lrf4}
\|Du^\ep\|_{L^{2(\ga-1)p}(\Tt^d)}\leq C\|D^2u^\ep\|^{\frac{1}{\ga}}_{L^{\frac{2(\ga-1)p}{\ga}}(\Tt^d)}\|u^\ep\|_{L^\infty(\Tt^d)}^{\frac{1}{2}}+C\|u^\ep\|_{L^\infty(\Tt^d)}^{\frac{1}{2}}  
\end{equation}
From \eqref{eq:lrf4} and Corollary \ref{c92} it follows that
\begin{align*}
\|Du^\ep\big\|_{L^F(0,T;L^G(\Tt^d))}
&\le  C\big\|D^2u^\ep\big\|^\frac{1}{\ga}_{L^\frac{F}{\ga}(0,T;L^\frac{G}{\ga}(\Tt^d))}\big\|u^\ep\big\|_{L^{\infty}(0,T;L^{\infty}(\Tt^d))}^{\frac{1}{2}}\\
&\quad+C\big\|u^\ep\big\|_{L^{\infty}(0,T;L^{\infty}(\Tt^d))}^{\frac{1}{2}}
\end{align*}and 
\begin{align*}
\big\|D^2u^\ep\big\|^\frac{1}{\ga}_{L^\frac{F}{\ga}(0,T;L^\frac{G}{\ga}(\Tt^d))}&\leq C\big\|g_\ep(m)\big\|^\frac{1}{\ga}_{L^{\frac{a_\upsilon}{\af}}(0,T;L^{\frac{b_\upsilon}{\af}}(\Tt^d))}\\&\quad+C\big\|u^\ep\big\|_{L^{\infty}(0,T;L^{\infty}(\Tt^d))}^{\frac{1}{2-\ga}}+C.
\end{align*}
By combining these, one obtains
\begin{align}\label{eq:eq3sec10}
\big\|Du^\ep\big\|_{L^F(0,T;L^G(\Tt^d))} &
\le C\big\|g_\ep(m)\big\|^\frac{1}{\ga}_{L^{\frac{a_\upsilon}{\af}}(0,T;L^{\frac{b_\upsilon}{\af}}(\Tt^d))}\big\|u^\ep\big\|_{L^{\infty}(0,T;L^{\infty}(\Tt^d))}^{\frac{1}{2}}\\
\notag &\quad +C\big\|u^\ep\big\|_{L^{\infty}(0,T;L^{\infty}(\Tt^d))}^{\frac{1}{2-\ga}+\frac{1}{2}}+C\big\|u^\ep\big\|_{L^{\infty}(0,T;L^{\infty}(\Tt^d))}^{\frac{1}{2}}.
\end{align}
Because of Lemma \ref{lem:newint} we also have
$$\big\|u^\ep\big\|_{L^{\infty}(0,T;L^{\infty}(\Tt^d))}\leq
C+C\big\|g_\ep\big\|_{L^{\frac{a_\upsilon}{\af}}(0,T;L^{\frac{b_\upsilon}{\af}}(\Tt^d))}^\zeta.$$
Hence, \eqref{eq:eq3sec10} becomes
\begin{align*}
\big\|Du^\ep\big\|&_{L^F(0,T;L^G(\Tt^d))}\leq C\big\|g_\ep\big\|^{\frac{1}{\ga}+\frac{\zeta}{2}}_{L^{\frac{a_\upsilon}{\af}}(0,T;L^{\frac{b_\upsilon}{\af}}(\Tt^d))}
\\&+C\big\|g_\ep\big\|_{L^{\frac{a_\upsilon}{\af}}(0,T;L^{\frac{b_\upsilon}{\af}}(\Tt^d))}^{\frac{\zeta}{2-\ga}+\frac{\zeta}{2}}
+C\big\|g_\ep\big\|_{L^{\frac{a_\upsilon}{\af}}(0,T;L^{\frac{b_\upsilon}{\af}}(\Tt^d))}^{\frac{\zeta}{2}}
\\&+C\big\|g_\ep\big\|_{L^{\frac{a_\upsilon}{\af}}(0,T;L^{\frac{b_\upsilon}{\af}}(\Tt^d))}^{\frac{1}{\ga}}+C\\
&\leq C\big\|g_\ep\big\|_{L^{\frac{a_\upsilon}{\af}}(0,T;L^{\frac{b_\upsilon}{\af}}(\Tt^d))}^{\frac{\zeta}{2-\ga}+\frac{\zeta}{2}}+C\big\|g_\ep\big\|_{L^{\frac{a_\upsilon}{\af}}(0,T;L^{\frac{b_\upsilon}{\af}}(\Tt^d))}^{\frac{1}{\ga}+\frac{\zeta}{2}}+C, 
\end{align*}
where the last inequality follows from Young's applied to those terms with lower exponents.
\end{proof}

\begin{Corollary}\label{cor1sec10}
Let $(u^\ep,m^\ep)$ be a solution of \eqref{eq:smfg} and assume that
A\ref{ah}-\ref{asbdpph} hold. Suppose further that
\eqref{eq:3sec10} holds. Then, 
\[
\big\|g_\ep(m)\big\|_{L^\frac{a_\upsilon}{\af}(0,T;L^\frac{b_\upsilon}{\af}(\Tt^d))}\leq 
C+C\big\||D_pH|^2\big\|_{L^r(0,T;L^p(\Tt^d))}^\frac{r\upsilon\af\left(1-\frac{1}{\theta}\right)}{\be_0(\theta-1)},
\]
where $p>\frac{d}{2}$ and $r$ is given by \eqref{eq:c4}.
\end{Corollary}
\begin{proof}
Lemma \ref{lemma2sec10} along with A\ref{ag2} leads to 
\[
\big\|g_\ep(m)\big\|_{L^\frac{a_\upsilon}{\af}(0,T;L^\frac{b_\upsilon}{\af}(\Tt^d))}
\le\big\|m^\ep\big\|^\af_{L^{a_\upsilon}(0,T;L^{b_\upsilon}(\Tt^d))}\leq
C+C\big\||D_pH|^2\big\|_{L^r(0,T;L^p(\Tt^d))}^\frac{r\upsilon\af\left(1-\frac{1}{\theta}\right)}{\be_0(\theta-1)}
\]
and the Corollary is established.
\end{proof}

\begin{Corollary}\label{cor2sec10}
Let $(u^\ep,m^\ep)$ be a solution of \eqref{eq:smfg} and assume that
A\ref{ah}-\ref{asbdpph} hold. 
Suppose further that \eqref{eq:3sec10}-\eqref{eq:6sec10} hold. 
Then, 
\begin{align*}
\big\|Du^\ep\big\|_{L^F(0,T;L^G(\Tt^d))}&\leq C+
C\big\|Du^\ep\big\|_{L^F(0,T;L^G(\Tt^d))}^{\frac{(1-\lam)(\ga-1)(4\zeta-\ga\zeta)}{(2-\ga)}\frac{r\upsilon\af\left(1-\frac{1}{\theta}\right)}{\be_0(\theta-1)}}\\&\quad+C\big\|Du^\ep\big\|_{L^F(0,T;L^G(\Tt^d))}^{\frac{(1-\lam)(\ga-1)(2+\ga\zeta)}{\ga}\frac{r\upsilon\af\left(1-\frac{1}{\theta}\right)}{\be_0(\theta-1)}}.
\end{align*}
where $p>\frac{d}{2}$ and $r$ is given by \eqref{eq:c4}.
\end{Corollary}
\begin{proof}
Lemma \ref{lemma3sec10} along with Corollary \ref{cor1sec10} leads to 
\begin{align*}
\big\|Du^\ep\big\|_{L^F(0,T;L^G(\Tt^d))}
&\leq  C+C\big\||D_pH|^2\big\|_{L^r(0,T;L^p(\Tt^d))}^{\frac{(4\zeta-\ga\zeta)}{2(2-\ga)}\frac{r\upsilon\af\left(1-\frac{1}{\theta}\right)}{\be_0(\theta-1)}}\\&\quad+C\big\||D_pH|^2\big\|_{L^r(0,T;L^p(\Tt^d))}^{\frac{2+\ga\zeta}{2\ga}\frac{r\upsilon\af\left(1-\frac{1}{\theta}\right)}{\be_0(\theta-1)}}
\end{align*}
Furthermore, because of Assumption A\ref{dphs} and Lemma \ref{lemma03sec10} we have
\[
\big\||D_pH|^2\big\|_{L^r(0,T;L^p(\Tt^d))}^{\frac{(4\zeta-\ga\zeta)}{2(2-\ga)}\frac{r\upsilon\af\left(1-\frac{1}{\theta}\right)}{\be_0(\theta-1)}}\leq C+C\big\|Du^\ep\big\|_{L^F(0,T;L^G(\Tt^d))}^{\frac{(1-\lam)(\ga-1)(2+2\zeta-\zeta\ga)}{(2-\ga)}\frac{r\upsilon\af\left(1-\frac{1}{\theta}\right)}{\be_0(\theta-1)}}
\]and
\[
\big\||D_pH|^2\big\|_{L^r(0,T;L^p(\Tt^d))}^{\frac{2+\ga\zeta}{2\ga}\frac{r\upsilon\af\left(1-\frac{1}{\theta}\right)}{\be_0(\theta-1)}}\leq C+C\big\|Du^\ep\big\|_{L^F(0,T;L^G(\Tt^d))}^{\frac{(1-\lam)(\ga-1)(2+\ga\zeta)}{\ga}\frac{r\upsilon\af\left(1-\frac{1}{\theta}\right)}{\be_0(\theta-1)}}.
\]The result follows by combining the former computation.
\end{proof}

\begin{Lemma}\label{lemma4sec10} 
Let $(u^\ep,m^\ep)$ be a solution of \eqref{eq:smfg}, assume that
A\ref{ah}-\ref{alphaimp} hold. 
Then,
\[
\big\|Du^\ep\big\|_{L^F(0,T;L^G(\Tt^d))}\leq C,
\]where $F$ and $G$ are given by \eqref{eq:eqF} and \eqref{eq:eqG}, respectively.
\end{Lemma}
\begin{proof}
By Corollary \ref{cor2sec10}, if \eqref{eq:1sec10}-\eqref{eq:6sec10} hold,
\begin{align*}
\big\|Du^\ep\big\|_{L^F(0,T;L^G(\Tt^d))}&\leq C+
C\big\|Du^\ep\big\|_{L^F(0,T;L^G(\Tt^d))}^{\frac{(1-\lam)(\ga-1)(4\zeta-\ga\zeta)}{(2-\ga)}\frac{r\upsilon\af\left(1-\frac{1}{\theta}\right)}{\be_0(\theta-1)}}\\&\quad+C\big\|Du^\ep\big\|_{L^F(0,T;L^G(\Tt^d))}^{\frac{(1-\lam)(\ga-1)(2+\ga\zeta)}{\ga}\frac{r\upsilon\af\left(1-\frac{1}{\theta}\right)}{\be_0(\theta-1)}}.
\end{align*}Also, 
\begin{align}\label{eq:7sec10}
\frac{(1-\lam)(\ga-1)(4\zeta-\ga\zeta)}{(2-\ga)}\frac{r\upsilon\af\left(1-\frac{1}{\theta}\right)}{\be_0(\theta-1)}&<1\\\label{eq:8sec10}\frac{(1-\lam)(\ga-1)(2+\ga\zeta)}{\ga}\frac{r\upsilon\af\left(1-\frac{1}{\theta}\right)}{\be_0(\theta-1)}&<1
\end{align}have to be satisfied.
The Lemma follows by combining Young's inequality with Lemma \ref{techlemma}.
\end{proof}
\begin{teo}\label{teosec10}
Let $(u^\ep,m^\ep)$ be a solution of \eqref{eq:smfg} and assume that
A\ref{ah}-\ref{alphaimp} hold. Then, for any $\be>1$,
$\|m^\ep\|_{L^{\infty}(0,T;L^{\be}(\Tt^d))}$ is bounded uniformly in $\ep$. 
\end{teo}
\begin{proof}
For  $p>\frac d2$, $\te>1$ and  $r$ given by Lemma \ref{lemma4sec10}, we
have by Theorem \ref{cpp2} that for any $\be>1$ there is $r_\be$ such that
\begin{align*}
\int_{\Tt^d}
(m^\ep)^{\be}(\tau,x)dx\leq C+C\big\||D_{p}H(x,Du^\ep)|^{2}\big\|_{L^{r}(0,T;L^{p}(\Tt^d))}^{r_{\be}}.
\end{align*}
By combining \eqref{eq:eq88plus} and Lemma \ref{lemma4sec10} with A\ref{dphs} one obtains
\[\||D_pH(x,Du^\ep)|^{2}\|_{L^{r}(0,T;L^{p}(\Tt^d))}\leq 
C\|Du^\ep\|_{L^F(0,T;L^G(\Tt^d))}^{2(\ga-1)(1-\lambda)}+C\le C,\]
which establishes the Theorem.
\end{proof}
\begin{Corollary}\label{sec10final}
Let $(u^\ep,m^\ep)$ be a solution of \eqref{eq:smfg}, assume that
A\ref{ah}-\ref{alphaimp} hold. Then, for any $p,r>1$,
$\|Du^\ep\|_{L^r(0,T;L^p(\Tt^d))}$, $\|D^2u^\ep\|_{L^r(0,T;L^p(\Tt^d))} $
are bounded uniformly in $\ep$. 
\end{Corollary}
\begin{proof}
By Theorem  \ref{teosec10}, for any $p,r>1$, $\|g_\ep(m^\ep)\|_{L^r(0,T;L^p(\Tt^d))}$ 
is bounded uniformly in $\ep$. So are $\big\|u^\ep\big\|_{L^\infty(0,T;L^\infty(\Tt^d))}$ and
$\|D^2u^\ep\|_{L^r(0,T;L^p(\Tt^d))}$, because of Lemma \ref{imphc0} and Corollary \ref{c92}.
Finally, from Lemma \ref{l91} 
\[\|Du^\ep\|_{L^{2r}(0,T;L^{2p}(\Tt^d))} \le C\|D^2u^\ep\|_{L^r(0,T;L^p(\Tt^d))}^\frac 12
\|u^\ep\|_{L^\infty(0,T;L^\infty(\Tt^d))}^\frac 12.\]
\end{proof}

\section{Lipschitz regularity}\label{lrsbh}

We now derive Lipschitz regularity for the solution $u^\epsilon$. To do so we use the non-linear adjoint method introduced in \cite{E3}, 
see also the applications in \cite{T1, ES, CGMT, CGT1, CGT2}.

\begin{teo}\label{teo1sec13}
Let $(u^\ep,m^\ep)$ be a solution of \eqref{eq:smfg}, assume that
A\ref{ah}-A\ref{alphaimp} hold. 
Then $Du^\epsilon\in L^\infty(\Tt^d\times\left[0,T\right])$, uniformly in $\epsilon$.
\end{teo}
\begin{proof}We omit the $\epsilon$.
Note that $u$ is a solution of the heat equation 
\begin{equation}\label{eq:heatf}
  \begin{cases}
    u_t+\Delta u&=f\\ u(x,T)&=\psi
  \end{cases}
\end{equation}
with $\psi\in W^{1,\infty}(\Tt^d)$ and $f\in L^a([0,T]\times\Tt^d)$
for any $a>1$. 
We introduce the adjoint equation 
\begin{equation}
\label{ADJ1}
\rho_t-\Delta \rho=0
\end{equation}
with initial data
$\rho(\cdot,\tau)=\delta_{x_0}$.
Multiplying \eqref{ADJ1} by $\nu\rho^{\nu-1}$ and integrating, we
have for $\tau<s<T$
\begin{equation}
\label{blabla}
\int_{\Tt^d}(\rho^\nu(x,T)-\rho^\nu(x,s))dx
=\int_s^T\int_{\Tt^d}\nu\rho^{\nu-1}\Delta\rho
=\frac{4(1-\nu)}{\nu}\int_s^T\int_{\Tt^d}|D(\rho^{\nu/2})|^2 dxdt.  
\end{equation}
Because $\rho(\cdot, t)$ is a probability measure and $0<\nu<1$ we have
$
\int_{\Tt^d} \rho^\nu(x,t) dx\leq 1
$.
Thus
\[
\int_\tau^T\int_{\Tt^d}|D\rho^{\nu/2}|^2dx\,dt\le \frac{\nu}{4(1-\nu)}.
\]

Fixing a unit vector $\xi\in\Rr^d$,
we have the following representation formula for the 
directional derivative $u_\xi$:  
\begin{equation}
\label{7A1}
u_\xi(x_0,\tau)-\int_{\Tt^d}\psi_\xi\rho(x,T)=-\int_\tau^T\int_{\Tt^d}f_\xi\rho=
\int_\tau\int_{\Tt^d}f\rho_\xi(x,T). 
\end{equation}
Note that
$
\left|\int_{\Tt^d} \psi_\xi\rho(x,T)\right|\le \|\psi\|_{W^{1,\infty}(\Tt^d)}.
$
For $0<\nu<1$, 
\begin{align*}
\left|\int_\tau^T\int_{\Tt^d}f\rho_\xi\right|\leq 
&\int_\tau^T\int_{\Tt^d}|f|\rho^{1-\frac\nu 2} |\rho^{\frac\nu 2-1}D\rho|\\
\le &\|f\|_{L^a([\tau,T]\times\Tt^d)}
\|\rho^{1-\frac \nu 2}\|_{L^b([\tau,T]\times\Tt^d)}
\|D\rho^{\frac \nu 2}\|_{L^2([\tau,T]\times\Tt^d)},
\end{align*}
for any $2\leq a, b\leq \infty$ satisfying
$
\frac 1 a +\frac 1 b+ \frac 1 2 =1 
$.
Therefore it suffices to bound
$
\|\rho^{1-\frac \nu 2}\|_{L^b([\tau,T]\times\Tt^d)}, 
$
for some $b>2$.

Let $\dfrac{d-1}d<\nu<1$, and $\kappa=\dfrac{d\nu}{d\nu+2}$. Then 
$1-\kappa+\frac{2\kappa}{2^*\nu}=\frac\kappa\nu$,
and therefore $1<\dfrac\nu\kappa<\dfrac{2^*\nu}2$. Moreover
$\dfrac\nu\kappa>2-\nu$. Define 
$b=\frac\nu{\kappa(1-\frac\nu 2)}>2$.
By H\"older inequality we have
\[
\Big(\int_{\Tt^d} \rho^{b( 1-\frac \nu 2)}\Big)^\frac{1}{b( 1-\frac \nu 2)}
=\Big(\int_{\Tt^d} \rho^{b( 1-\frac \nu 2)}\Big)^{\frac\kappa\nu}
\leq \Big(\int_{\Tt^d} \rho\Big)^{1-\kappa}
\Big(\int_{\Tt^d} \rho^{\frac{2^*\nu}{2}}\Big)^{\frac{2\kappa}{2^*\nu}}. 
\]
Recall that by Sobolev's inequality we have
$
\left(\int_{\Tt^d} \rho^{\frac{2^*\nu}{2}}\right)^{\frac{2}{2^*}}\leq
C+C \int_{\Tt^d} |D\rho^{\frac \nu 2}|^2$. 
Therefore 
\[\int_{\Tt^d}\rho^{b( 1-\frac \nu 2)}\leq
C+C \int_{\Tt^d} |D\rho^{\frac \nu 2}|^2,\]
and then
\[
\int_\tau^T\int_{\Tt^d}\rho^{b( 1-\frac \nu 2)}
\le C+C \int_\tau^T\int_{\Tt^d} |D\rho^{\frac \nu 2}|^2 \le C.
\]
\end{proof}

\section{Improved regularity for the Fokker-Planck equation}
\label{irfp}

We proceed to obtain further regularity and integrability on $m^\epsilon$ and 
$\eta_\epsilon*m^\epsilon$ building upon the results in the previous
sections. For convenience of notation we will consider the convolution
$\eta_\delta*m^\epsilon$ where $\delta\in\{0, \epsilon\}$ with the natural
convention that $\eta_0*m^\epsilon=m^\epsilon$. 
We start with a modified version of 
Proposition \ref{ent-prod}:

\begin{Proposition} \label{ent-prod2}
Let $(u^\epsilon,m^\epsilon)$ be a solution of \eqref{eq:smfg}. 
Let $\varphi:\Rr\to \Rr$ be a $C^2$ function, and let $\delta\in\{0, \epsilon\}$. Then
\begin{align*}
\frac{d}{dt}\int_{\Tt^d} \varphi(\eta_\delta* m^\epsilon) dx&
+\int_{\Tt^d}\varphi''(\eta_\delta* m^\epsilon) 
D_x(\eta_\delta* m^\epsilon)\cdot\eta_\delta *(m^\epsilon D_pH)dx  \\
&=-\int_{\Tt^d} \varphi''(\eta_\delta* m^\epsilon) |D_x (\eta_\delta* m^\epsilon)|^2 dx.
\end{align*} 
\end{Proposition}
\begin{proof}
The proof follows by first convolving the second equation in \eqref{eq:smfg}
with $\eta_\delta$, then 
by multiplying it by $\varphi'(\eta_\delta*m^\epsilon)$ and finally 
integrating in $\Tt^d$. 
\end{proof}

\begin{Proposition}
\label{firstreg}
Let $(u^\epsilon,m^\epsilon)$ be a solution of \eqref{eq:smfg}. Suppose that A\ref{ah}-\ref{alphaimp} hold.  Let $\delta\in\{0, \epsilon\}$.
For any $\be>0$, $\eta_\delta* m^\epsilon\in L^{\infty}((0,T), L^\be(\Tt^d))$, 
$D_x((\eta_\delta* m^\epsilon)^\be)\in L^2(\Tt^d\times\left[0,T\right])$.  
Additionally $D_x\ln(\eta_\delta* m^\epsilon)\in L^2(\Tt^d\times\left[0,T\right])$.  
Furthermore, all the bounds are uniform in $\epsilon$ and $\delta$.
\end{Proposition}
\begin{proof}
Note that for $0<\be\leq 1$ the statement 
$\eta_\delta* m^\epsilon\in L^{\infty}((0,T),
L^\be(\Tt^d))$ is immediate, so we discuss only the case $\be>1$. 
Since $D_xu^\epsilon$ is bounded, so is $|D_pH(x,D_xu^\epsilon)|$. Thus
applying Proposition \ref{ent-prod2} for $\varphi(z)=z^\be$ yields
\begin{align}\notag
& \frac{d}{dt}\int_{\Tt^d}(\eta_\delta* m^\epsilon)^\be dx
=-\be(\be-1)\int_{\Tt^d} (\eta_\delta* m^\epsilon)^{\be-2}
\eta_\delta*(m^\epsilon D_pH)D_x(\eta_\delta* m^\epsilon)\\
&-\be(\be-1)\int_{\Tt^d}(\eta_\delta* m^\epsilon)^{\be-2}
|D_x(\eta_\delta*m^\epsilon)|^2 dx\\\notag
&\le \frac{\be(\be-1)}2\int_{\Tt^d} (\eta_\delta* m^\epsilon)^{\be-2}
\left(|\eta_\delta*(m^\epsilon D_pH)|^2-
|D_x(\eta_\delta* m^\epsilon)|^2\right)dx\\\label{est}
&\le C\int_{\Tt^d}(\eta_\delta* m^\epsilon)^\be dx
-2\frac{\be-1}\be\int|D_x(\eta_\delta* m^\epsilon)^{\be/2}|^2dx.
\end{align}
Estimate \eqref{est} follows from the fact that, because
$|D_pH(x,Du_\epsilon)|\le C$ and $m^\epsilon\geq 0$, we have
\begin{equation}
  \label{eq:simple}
|\eta_\delta*(m^\epsilon D_pH)|\le C \eta_\delta* m^\epsilon.  
\end{equation}

Consequently $\eta_\delta* m^\epsilon\in L^{\infty}((0,T), L^\be(\Tt^d))$.
Integrating \eqref{est} in time we get
\begin{align*}
&2\frac{\be-1}\be\int_0^T\int_{\Tt^d}|D_x(\eta_\delta*m^\epsilon)^{\be/2}|^2dxdt
\le C\int_0^T\int_{\Tt^d}(\eta_\delta*m^\epsilon)^\be dx dt\\&
+\int_{\Tt^d}(\eta_\delta* m^\epsilon)^\be(x,0) dx
-\int_{\Tt^d}(\eta_\delta* m^\epsilon)^\be(x,T) dx. 
\end{align*}
Thus $D_x(\eta_\delta* m^\epsilon)^\ga\in L^2((0,T)\times\Tt^d)$ 
for $\ga>\dfrac 12$. 

Now Proposition \ref{ent-prod2} applied to $\varphi(z)=\ln z$ yields
\[
\frac{d}{dt}\int_{\Tt^d}\ln(\eta_\delta*m^\epsilon)dx
-\int_{\Tt^d}\frac{\eta_\delta*(m^\epsilon D_pH)D_x(\eta_\delta*m^\epsilon)}
{(\eta_\delta*m^\epsilon)^2}dx=
\int_{\Tt^d}\frac{|D_x(\eta_\delta*m^\epsilon)|^2}{(\eta_\delta*m^\epsilon)^2}dx, 
\]
from which it follows using \eqref{eq:simple}
\[
\int_0^T \int_{\Tt^d} \frac{|D_x(\eta_\delta*m^\epsilon)|^2}
{(\eta_\delta*m^\epsilon)^2} dx \leq
C+ 2\int_{\Tt^d}\ln(\eta_\delta*m^\epsilon)(x,T)dx 
- 2\int_{\Tt^d} \ln(\eta_\delta*m^\epsilon)(x, 0) dx.
\]
By Jensen's inequality 
\[\int_{\Tt^d}\ln(\eta_\delta*m^\epsilon)(x,T)dx
\leq\ln\int_{\Tt^d}\eta_\delta*m^\epsilon(x,T)dx
=\ln\int_{\Tt^d}m^\epsilon(x,T)dx=0.\] 
Therefore $D_x\ln(\eta_\delta*m^\epsilon)\in L^2(\Tt^d\times\left[0,T\right])$.

Let now $0<\be\le \frac 12$. For $q>\dfrac 1{2\be}$ let $p$ be the conjugated
exponent. Since $\be-1=\frac{q\be-1}q-\frac 1p$, we have
\begin{align}
\label{gg0}
|D_x(\eta_\delta*m^\epsilon)^\be|
&=(\tfrac 1q|D_x(\eta_\delta*m^\epsilon)^{q\be}|)^{\frac 1q}
(\be|D_x\ln(\eta_\delta*m^\epsilon)|)^{\frac 1p}.
\end{align} 
Now we observe that by the previous two parts of the proof 
\[(\tfrac 1q|D_x(\eta_\delta*m^\epsilon)^{q\be}|)^{\frac 1q}
\in L^{2q}((0,T)\times \Tt^d),\quad
(\be|D_x\ln(\eta_\delta*m^\epsilon)|)^{\frac 1p}
\in L^{2p}((0,T)\times\Tt^d),\] 
from which we conclude 
$D_x (\eta_\delta*m^\epsilon)^\be \in L^2(\Tt^d\times\left[0,T\right])$.
\end{proof}

\begin{Corollary} \label{freg}
Let $(u^\epsilon,m^\epsilon)$ be a solution of \eqref{eq:smfg}. Suppose that A\ref{ah}-\ref{alphaimp} hold. Then $D_{xx}^2u^\epsilon, u_t^\epsilon\in L^r([0,T]\times\Tt^d)$,
for any $r<\infty$. Furthermore, all the bounds are uniform in $\epsilon$.
 \end{Corollary}
 \begin{proof}
Let $f=H(x,Du^\epsilon)-\eta_\epsilon*(\eta_\epsilon*m^\epsilon)^\af$. 
Then $f\in L^\infty([0,T],L^r(\Tt^d))$ by Theorem \ref{teo1sec13} 
and Proposition \ref{firstreg}. Since 
$L^\infty([0,T],L^r(\Tt^d)) \subset L^r([0,T]\times\Tt^d)$, 
it follows by the regularity theory for the heat equation that
$D_{xx}^2u^\epsilon, u^\epsilon_t\in L^r([0,T]\times\Tt^d)$.
\end{proof}

 \begin{Corollary}\label{uregular}
Let $(u^\epsilon,m^\epsilon)$ be a solution of \eqref{eq:smfg}. Suppose that A\ref{ah}-\ref{alphaimp} hold. Then   
$D_{xxx}^3u^\epsilon, D_{xt}^2u^\epsilon\in L^2([0,T]\times\Tt^d)$,
$D_{xx}u^\epsilon\in L^\infty([0,T], L^2(\Tt^d))$. Furthermore, all
the bounds are uniform in $\epsilon$. 
 \end{Corollary}
\begin{proof}
Let $f=H(x,Du^\epsilon)-\eta_\epsilon*(\eta_\epsilon*m^\epsilon)^\af$
as in the proof of Corollary \ref{freg}. 
We claim that  
$D_xf\in L^2([0,T]\times\Tt^d)$.
Indeed, since
\[
D_x(H(x,Du^\epsilon))=D_xH+D_pHD^2_{xx}u^\epsilon\in L^r([0,T]\times\Tt^d), 
\]
and so 
using Proposition \ref{firstreg} we conclude that 
$D_xf\in L^2([0,T]\times\Tt^d)$.
 
The result then follows by applying regularity theory to 
$
u_{x_it}^\epsilon+\Delta u_{x_i}^\epsilon=f_{x_i}$.
\end{proof}

\begin{Corollary}\label{mregular}
Let $(u^\epsilon,m^\epsilon)$ be a solution of \eqref{eq:smfg}. Suppose that A\ref{ah}\ref{alphaimp} hold. Then 
   $D_{xx}^2m^\epsilon, m_t^\epsilon\in L^2([0,T]\times\Tt^d)$, and
   $D_xm^\epsilon\in L^\infty([0,T], L^2(\Tt^d))$.
 \end{Corollary}
 \begin{proof}
We have
$
  D_{x_i}(D_{p_j}H(x,D_xu^\epsilon))=D^2_{x_ip_j}H+D^2_{p_jp_k}H D^2_{x_ix_k}u^\epsilon
$
  from what it follows, using Corollaries \ref{freg} and \ref{uregular}, that for any $r<\infty$
\begin{equation}
  \label{eq:div}
D_{x}(D_{p}H(x,D_xu^\epsilon)) \in L^{r}([0,T]\times\Tt^d) \cap L^\infty([0,T], L^2(\Tt^d)).
  \end{equation}
We have
\begin{equation}
\label{divexp}
\div(D_pHm^\epsilon)=\div(D_pH)m^\epsilon+D_pH\cdot D_xm^\epsilon
\end{equation}
Since $m^\epsilon\in L^\infty([0,T], L^r(\Tt^d))$, for all $r>1$, by Proposition \ref{firstreg}, 
using \eqref{eq:div} we get $\div(D_pH)m^\epsilon\in L^2(\Tt^d\times\left[0,T\right])$. 
Also, by proposition \ref{firstreg}, we have $D_x m^\epsilon\in L^2(\Tt^d\times\left[0,T\right])$, 
and so $D_pH\cdot D_xm^\epsilon \in L^2(\Tt^d\times\left[0,T\right])$. 
Therefore
$
\div(D_pHm^\epsilon)\in L^2([0,T]\times\Tt^d)
$.
Applying regularity theory to 
$
m_t^\epsilon-\Delta m^\epsilon=\div(D_pHm^\epsilon)
$
we get the Corollary.
\end{proof}

 \begin{Corollary}\label{mholder}
Let $(u^\epsilon,m^\epsilon)$ be a solution of \eqref{eq:smfg}. Suppose that A\ref{ah}-\ref{alphaimp} hold. There is $\br>d$ such that 
$D_xm^\epsilon, D_{xx}^2m^\epsilon, m_t^\epsilon\in L^{\br}(\Tt^d\times\left[0,T\right])$ and then
$m^\epsilon\in C^{0,1-d/\br}(\Tt^d\times\left[0,T\right])$.
 \end{Corollary}
 \begin{proof}
We define inductively a finite increasing sequence
$\be_0=2,\ldots,\be_{N-1}\le d<\be_N$ such that the 
statement of the Corollary holds for $r=\be_n$ and $\be_n\le d$. Assuming 
$\be_n\le d$ is already defined let $\be_{n+1}=(1+\frac 2d)\be_n$. As usually, 
denote by $\be_n^*$ the Sobolev conjugated exponent to $\be_n$: $
\frac{1}{\be_n^*}=\frac 1 {\be_n} -\frac 1 d$.
By H\"older and Sobolev inequalities
\[
\int_{\Tt^d}|D_xm^\epsilon|^{\be_{n+1}}\le\left(\int_{\Tt^d}
  |D_xm^\epsilon|^{\be_n^*}\right)^{\frac{\be_n}{\be_n^*}}  
\left(\int_{\Tt^d}|D_xm^\epsilon|^2\right)^{\frac{\be_n}d}
\le\|m^\epsilon\|_{W^{2,\be_n}}^{\be_n}\|D_xm^\epsilon\|_2^{\frac {2\be_n}d}.
 \]
Integrating in the time variable
\[
\|D_xm^\epsilon\|^{\be_{n+1}}_{L^{\be_{n+1}}(\Tt^d\times\left[0,T\right])}
\le \|m^\epsilon\|_{L^{\be_n}(0,T; W^{2,\be_n})}^{\be_n}
\|D_xm^\epsilon\|_{L^\infty([0,T], L^2(\Tt^d)}^{\frac{2\be_n}d}.
\]
Taking into account \eqref{divexp}, and the fact that 
 $\div(D_pH) \in L^p(\Tt^d\times\left[0,T\right])$, for any $p>1$,
 $m^\epsilon\in L^\infty([0,T], L^r(\Tt^d))$ for any $r>1$,  
$D_pH$ bounded, and the above estimate for $D_xm$, 
we conclude that  
 $\div(D_pHm^\epsilon)\in L^{\be_{n+1}}(\Tt^d\times\left[0,T\right])$.
The statement of the Corollary for $r=\be_{n+1}$ then follows by standard parabolic regularity. The H\"older
continuity is a consequence of Morrey's theorem. 
\end{proof}

\begin{Corollary} \label{m+regular}
Let $(u^\epsilon,m^\epsilon)$ be a solution of \eqref{eq:smfg}. Suppose that A\ref{ah}-\ref{alphaimp} hold. Then $D_{xxx}^3m^\epsilon, D^2_{xt}m^\epsilon\in L^2(\Tt^d\times\left[0,T\right])$, 
$D^2_{xx}m^\epsilon\in L^{\infty}([0,T], L^2(\Tt^d))$.
 \end{Corollary}
 \begin{proof}
We start by establishing that 
 $D_{xx}(D_pHm^\epsilon)\in L^2(\Tt^d\times\left[0,T\right])$.
 

Note that 
\begin{align}\notag
 D^2_{x_ix_k}(D_{p_l}Hm^\epsilon)&=D^2_{x_ix_k}(D_{p_l}H)m^\epsilon+D_{x_i}(D_{p_l}H)D_{x_k}m^\epsilon\\
&+D_{x_k}(D_{p_l}H)D_{x_i}m^\epsilon+D^2_{x_ix_k}m^\epsilon(D_{p_l}H). 
  \label{eq:dd(mdiv)}
\end{align}
The first term $D^2_{x_ix_k}(D_{p_l}H)m^\epsilon$, since $m^\epsilon$ is bounded by Corollary \ref{mholder}, can be estimated by 
showing that 
\begin{align}\notag
 &D^2_{x_ix_k}(D_{p_l}H(x,Du^\epsilon))=D_{x_ix_kp_j}^3H+D_{x_ip_jp_l}^3Hu_{x_lx_k}^\epsilon
+D_{p_jp_lx_k}^3Hu_{x_lx_i}^\epsilon\\\label{ddiv}
&\quad +D_{p_jp_lp_m}^3Hu_{x_lx_i}^\epsilon u_{x_mx_k}^\epsilon+D^2_{p_jp_l}Hu_{x_lx_ix_k}^\epsilon
\in L^{2}(\Tt^d\times\left[0,T\right]). 
\end{align}
In the previous equation $D_{x_ix_kp_j}^3H$ is bounded; $D_{x_ip_jp_l}^3Hu_{x_lx_k}^\epsilon, \ D_{p_jp_lx_k}^3Hu_{x_lx_i}^\epsilon
\in L^{2}(\Tt^d\times\left[0,T\right])$, by Corollary \eqref{freg}; using that $D^2_{xx} u^\epsilon\in L^r(\Tt^d\times\left[0,T\right])$, 
for any $r<\infty$ we also get that $D_{p_jp_lp_m}^3Hu_{x_lx_i}^\epsilon u_{x_mx_k}^\epsilon \in L^{2}(\Tt^d\times\left[0,T\right])$; 
finally $D^2_{p_jp_l}Hu_{x_lx_ix_k}^\epsilon$ is controlled thanks to Corollary \ref{uregular}. 

The second term $D_{x_i}(D_{p_l}H)D_{x_k}m^\epsilon$ and the 
 third term $D_{x_k}(D_{p_l}H)D_{x_i}m^\epsilon$ can be estimated by observing that $D_xm^\epsilon \in L^\br(\Tt^d\times\left[0,T\right])$, for
some $\br>d\geq 2$, by Corollary \ref{mholder} and using \eqref{eq:div} in the proof Corollary \ref{mregular} which states that 
$D_x(D_pH) \in L^r(\Tt^d\times\left[0,T\right])$, for any $r<\infty$. Hence, by taking $r$ large enough we have
$
D_{x_i}(D_{p_l}H)D_{x_k}m^\epsilon,\  D_{x_k}(D_{p_l}H)D_{x_i}m^\epsilon\in L^2(\Tt^d\times\left[0,T\right])$.
Finally the last term $D^2_{x_ix_k}m^\epsilon(D_{p_l}H)$, since $D_pH$ is bounded can be controlled by the estimate 
for $D^2_{xx} m^\epsilon$ in Corollary \ref{mregular}. And so $D^2_{x_ix_k}m^\epsilon(D_{p_l}H)\in L^2(\Tt^d\times\left[0,T\right])$.
From which we conclude that $ D^2_{x_ix_k}(D_{p_l}Hm^\epsilon)\in L^2(\Tt^d\times\left[0,T\right])$. 
The proof ends by applying
standard regularity theory to 
$
m_{x_it}^\epsilon-\Delta m_{x_i}^\epsilon=(\div(D_pHm^\epsilon))_{x_i}.
$
 \end{proof}

\section{Further regularity by the Hopf-Cole transformation}
\label{frhct}
To simplify the notation, throughout this Section we will omit the $\epsilon$ as all the estimates previously obtained are uniform in $\epsilon$. 
We use the log transform $w=\ln m$, then $w$ solves
\begin{equation}
  \label{eq:logm}
w_t=\div(D_pH(x,D_xu)) +D_pH(x,D_xu) Dw+|Dw|^2+\Delta w.  
\end{equation}
Therefore
$w$ is the value function for
the following stochastic optimal control problem 
\begin{align*}
  w(x,t)=&\sup_\bv E
  \Big[
  (w(\bx(0),0)\\
  &+\int_0^t(\div D_pH)(\bx,D_xu(\bx,s)) 
-\frac {|\bv-D_pH(\bx,D_xu(\bx,s))|^2}4ds)\Big],
\end{align*}
where the supremum is taken over all bounded and progressively
measurable controls $\bv$,
\[
d\bx=\bv ds+\sqrt{2} dW_{t-s},\quad   \bx(t)=x
\] 
and $W_s$ is a $d$-dimensional Brownian motion.

This implies the following Lax-Hopf estimate: 
\begin{align}
\label{lhw}
w(x,t)\geq&\int_0^t\int_{\Tt^d}\big(\div(D_pH)-\frac 14|b-D_pH|^2\big)
\zeta(x,y,s) dy ds\\\notag &+\int_{\Tt^d} w(y,0)\zeta(x,y,0), 
\end{align}
where $b$ is any smooth vector field $b:\Tt^d\times \left[0,t\right]\to \Rr^d$,
and $\zeta$ is a solution of 
\begin{equation}
\label{dlaw}
\zeta_s+\Delta\zeta=\div(b\zeta),
\end{equation}
with $\zeta(x,y,t)=\delta_x(y)$. 

\begin{Corollary}
\label{imphk}
Let $(u^\epsilon,m^\epsilon)$ be a solution of \eqref{eq:smfg}. Suppose that A\ref{ah}-\ref{alphaimp} hold. Set $w=\ln m$.   
If $r>d$ we have 
\begin{equation}
\label{atron}
w(x,t)\ge-C-C\|\div(D_pH)\|_{L^2(0,t;L^r(\Tt^d))}+\int_{\Tt^d}
w(x,0)\theta(y,t)dy, 
\end{equation}
where $\theta(y,s)$ is the heat kernel with $\theta(\cdot,0)=\delta_{x}$, and $q$
is the conjugated exponent to $r$, $\frac 1 q + \frac 1 r =1$. 
\end{Corollary}
\begin{proof}
Take $b=0$ in \eqref{lhw} and recall that $D_pH(x,D_xu)$ is bounded.
\begin{align*}
-\int_{\Tt^d}\div(D_pH)\theta(y,t-s)dy&
\le \|\div(D_pH(\cdot,D_xu(\cdot,s))\|_r\|\theta(\cdot,t-s)\|_q\\
&\le C\|\div(D_pH(\cdot,D_xu(\cdot,s))\|_r(t-s)^{-\frac d{2r}}.
\end{align*}
Thus
\[-\int_0^t\int_{\Tt^d}\div(D_pH)\theta(y,t-s)dy ds\le
C\|\div(D_pH)\|_{L^2(0,t;L^r(\Tt^d))}.\]
\end{proof}

The optimal drift $b$ in \eqref{lhw} is given by the 
following Proposition 

\begin{Proposition}
Let $(u^\epsilon,m^\epsilon)$ be a solution of \eqref{eq:smfg}. Suppose that A\ref{ah}-\ref{alphaimp} hold. Set $w=\ln m$. Consider the solution $\rho$ to the adjoint equation
\begin{equation}
\label{ADJw}
\rho_t+\Delta \rho =\div((D_pH+2Dw)\rho)
\end{equation}
with terminal data $\rho(\cdot,t)=\delta_{x_0}$. Then
\begin{equation}
\label{7z}
w(x_0,t)=\int_0^t\int_{\Tt^d}(\div D_pH-|Dw|^2)\rho+\int_{\Tt^d}w(x,0)\rho(x,0). 
\end{equation}
\end{Proposition}
\begin{proof}
Multiply \eqref{eq:logm} by $\rho$ and \eqref{ADJw} by $w$ and integrate.  
Then using integration by parts we
obtain \eqref{7z}.  
\end{proof}
\begin{Corollary}
\label{c8p}
Let $(u^\epsilon,m^\epsilon)$ be a solution of \eqref{eq:smfg}. Suppose that A\ref{ah}-\ref{alphaimp} hold. Set $w=\ln m$ and let $\rho$ solve \eqref{ADJw}.  
Suppose $r>d$. Let $q$ be the conjugate exponent to $r$, that is
$\frac 1r+\frac 1q=1$.
Then
\begin{equation}
\label{pqr1}
\int_0^t \int_{\Tt^d} |Dw|^2 \rho \leq C+C\|\rho\|_{L^2(0,t;L^q(\Tt^d))}. 
\end{equation}
\end{Corollary}
\begin{proof}
Recalling the estimate \eqref{atron} from  Corollary \ref{imphk} 
\begin{align*}
&\int_0^t\int_{\Tt^d}|Dw|^2\rho=\int_0^t\int_{\Tt^d}\div D_pH\rho+
\int_{\Tt^d}w(x,0)\rho(x,0)-w(x_0,t)\\
&\le C+C\|\div(D_pH)\|_{L^2(0,t;L^r(\Tt^d))}\|\rho\|_{L^2(0,t;L^q(\Tt^d))}+
\int_{\Tt^d}w(x,0)\rho(x,0)dx.
\end{align*}The result follows from \eqref{eq:div}.
\end{proof}
\begin{Proposition}
\label{p8p3}
Let $(u^\epsilon,m^\epsilon)$ be a solution of \eqref{eq:smfg}. Suppose that A\ref{ah}-\ref{alphaimp} hold. Set $w=\ln m$ and let $\rho$ solve \eqref{ADJw}.  
For $\nu<1$
\[
\|D \rho^{\nu/2}\|_{L^2(0,t;L^2(\Tt^d))}^2\leq C 
+\delta_1\int_0^t \int_{\Tt^d} |Dw|^2\rho dx dt.
\]
\end{Proposition}
\begin{proof}
Multiply \eqref{ADJw} by $\nu\rho^{\nu-1}$ and integrate by parts to obtain $$\frac{d}{dt}\int_{\Tt^d}\rho^\nu dx +\frac{4(1-\nu)}{\nu}\int_{\Tt^d}|D\rho^{\frac{\nu}{2}}|dx=\nu(1-\nu)\int_{\Tt^d}(D_pH+2Dw)\rho^{\nu-1}D\rho\, dx.$$Since $0<\nu<1$ and because $\rho$ is a probability measure, the former computation implies
\begin{align*}
\int_0^T\int_{\Tt^d}|D\rho^{\frac{\nu}{2}}|\,dxdt&\leq C+C\int_0^T\int_{\Tt^d}|(D_pH+2Dw)|\rho^{\nu-1}|D\rho|\, dx\\&\leq C+C\int_0^T\int_{\Tt^d}|D_pH|\rho^{\nu-1}|D\rho| +C\int_0^T\int_{\Tt^d}|Dw|\rho^{\nu-1}|D\rho|\\&\leq C+C\int_0^T\int_{\Tt^d}|D_pH|^2\rho^\nu +C\int_0^T\int_{\Tt^d}|Dw|^2\rho^\nu\\&\quad+\delta\int_0^T\int_{\Tt^d}|D\rho^{\frac{\nu}{2}}|\,dxdt,
\end{align*}where the last inequality follows from a weighted Cauchy inequality. The result follows the Proposition \ref{firstreg} combined with the facts that $u$ is Lipschitz, $\rho$ is a probability measure, $0<\nu<1$ and $\rho^\nu\leq C_{\delta_1}+\delta_1\rho$.
\end{proof}

\begin{Corollary}\label{cor:cor153}
Let $(u^\epsilon,m^\epsilon)$ be a solution of \eqref{eq:smfg}. Suppose that A\ref{ah}-\ref{alphaimp} hold. Set $w=\ln m$ and let $\rho$ solve \eqref{ADJw}.  
Assume $r>d$. Let $q$ be the conjugate exponent to $r$, that is
$\frac 1r+\frac 1q=1$. Then
\[
\int_0^t\int_{\Tt^d}|D\rho^{\nu/2}|^2dx\,dt
\leq C+C\delta\|\rho\|_{L^2(0,t;L^q(\Tt^d))}.
\]
\end{Corollary}
\begin{proof}
This result follows by using Corollary \ref{c8p} in the estimate given by Proposition \ref{p8p3}.
\end{proof}

\begin{Proposition}\label{prop:prop153}
Let $(u^\epsilon,m^\epsilon)$ be a solution of \eqref{eq:smfg}. Suppose that A\ref{ah}-\ref{alphaimp} hold. Set $w=\ln m$ and let $\rho$ solve \eqref{ADJw}.  
Set
$
\mu=\frac{1-\frac 1 q}{\nu-\frac 2 {2^*}}
$.
If $2\mu<1$, 
then
\[
\|\rho\|_{L^2(0,t;L^q(\Tt^d))}\leq C+C \|D\rho^{\nu/2}\|_{L^2(0,t;L^2(\Tt^d))}^{2\mu}. 
\]
\end{Proposition}
\begin{proof}
Recall that for any $1\leq p_0< p_1<\infty,\, 0<\theta<1$ we have 
$
\|f\|_{p_{\theta}}\leq\|f\|_{p_1}^\theta\|f\|_{p_0}^{1-\theta},
$
where $p_{\theta}$ is given by
$
\frac{1}{p_{\theta}}=\frac{\theta}{p_1}+\frac{1-\theta}{p_0}
$.
Let
$p_1=1,$ $p_0=\frac{2^*\nu}{2}$. Then for $p_{\theta}=q$, we have
$
1-\theta=\frac{1-\frac{1}{q}}{1-\frac{1}{p_0}}.
$
By Sobolev's inequality
$
\left(\int_{\Tt^d}\rho^{\frac{2^*\nu}{2}}(x,t)\right)^{\frac{1}{2^*}}
\leq C+C\left(\int_{\Tt^d}|D(\rho^{\frac{\nu}{2}})(x,t)|^2dx \right)^{\frac{1}{2}},
$
and so
$
\|\rho(\cdot,t)\|_{\frac{2^*\nu}{2}}\leq C +
C\|D(\rho^{\frac{\nu}{2}})(\cdot,t)\|_2^{\frac{2}{\nu}}.
$
Using $\|\rho(.,t)\|_1=1$ and we conclude that 
\[
\|\rho(\cdot,t)\|_q\leq C+
C\|D(\rho^{\frac{\nu}{2}})(\cdot,t)\|_2^{2\mu}, 
\]
with $\mu=\frac{1-\theta}{\nu}.$ 
Then, if $2\mu<1$, by Jensen's inequality
\[
\|\rho\|^2_{L^2(0,t;L^q(\Tt^d))}
\leq C+C\|D(\rho^{\frac{\nu}{2}})\|_{L^2(0,t;L^2(\Tt^d))}^{4\mu}. 
\]
\end{proof}

{\bf Remark.} Let $r$ be the conjugate exponent to $q$, $\frac 1r+\frac 1 q=1$. 
Then  
$$\mu=\frac{1-\frac{1}{q}}{\nu-\frac{2}{2^*}}=\frac{1}{r(\nu-\frac{2}{2^*})}
=\frac{1}{r(\nu-1+\frac{2}{d})}.$$  
For $\nu$ close to $1$ and $r>d,$ we have $2\mu<1.$ In general $2\mu< 1$ provided
\begin{equation}
\label{nubound}
\nu > \frac 2r +\frac 2{2^*}. 
\end{equation}

\begin{Corollary}
\label{c8p14}
Let $(u^\epsilon,m^\epsilon)$ be a solution of \eqref{eq:smfg}. Suppose that A\ref{ah}-\ref{alphaimp} hold. Set $w=\ln m$ and let $\rho$ solve \eqref{ADJw}.  
Suppose $r>d$. Let $q$ be the conjugate exponent to $r$, that is 
$\frac 1r+\frac 1q=1$, and suppose \eqref{nubound} holds. Then
\[
\|\rho\|_{L^2(0,T;L^q(\Tt^d))}\leq C, 
\qquad 
\int_0^t \int_{\Tt^d} |Dw|^2 \rho \leq C,
\]
and
\[
\|D \rho^{\nu/2}\|_{L^2(0,T;L^2(\Tt^d))}^2 \leq C.
\]
\end{Corollary}
\begin{proof}
The assertions of the Corollary follow by combining Corollary \ref{c8p}, Proposition \ref{prop:prop153} and Corollary \ref{cor:cor153} with elementary inequalities.
\end{proof}

\begin{teo}\label{mLip}
Let $(u^\epsilon,m^\epsilon)$ be a solution of \eqref{eq:smfg}. Suppose that A\ref{ah}-\ref{alphaimp} hold. Set $w=\ln m$, and let $\rho$ solve \eqref{ADJw}. Then, 
$\ln m$ is Lipschitz, and hence $m$ is bounded by above and below. 
  \end{teo}
  \begin{proof}
Differentiate equation \eqref{eq:logm} with respect to a direction
determined by a unit vector $\xi\in \Rr^d$. Then 
\[w_{\xi t}=(\div D_pH(x,D_xu))_\xi+\left[D_pH(x,D_xu)\right]_\xi Dw+
(D_pH(x,D_xu)+2Dw) Dw_\xi+\Delta w_\xi.\]
Integrate the previous identity with respect to $\rho$, then
\begin{equation}
  \label{eq:dw}
w_\xi(x_0,t)=\int_0^t\int_{\Tt^d}((\div D_pH)_\xi+(D_pH)_\xi Dw)\rho\,dx\,ds
+\int_{\Tt^d}w_\xi(x,0)\rho(x,0)dx. 
\end{equation}
The proof ends if we manage to obtain an uniform bound in the integral terms of \eqref{eq:dw}. We first address the term
$
\int_0^t\int_{\Tt^d}(\div D_pH)_\xi \rho 
$.
As previously, we integrate by parts and therefore it suffices to bound
$
\left| \int_0^t\int_{\Tt^d}(\div D_pH) \rho_\xi\right|$.
Recall also, that by \eqref{eq:div} we have 
$\div D_pH \in L^p(\Tt^d\times\left[0,T\right])$ for any $p<\infty$.
We have, for $\nu$ close enough to $1$,
\begin{align*}
\left| \int_0^t\int_{\Tt^d}(\div D_pH) \rho_\xi\right|&\leq
C\int_0^t\int_{\Tt^d} |\div D_pH| \rho^{1-\frac \nu 2} D\rho^{\frac \nu 2}\\
&\leq
C\left(\int_0^t\int_{\Tt^d} |\div D_pH|^2 \rho^{2-\nu}\right)^{1/2}.
\end{align*}
This term is bounded provided we can show that, for some $\nu$ close enough to $1$
we have $\rho^{2-\nu}\in L^s$, for some $s>1$. It is indeed the case
since $\rho \in L^\infty([0,T], L^1)\cap L^2([0,T],L^q)$.

Secondly we consider the term 
$
\int_0^t\int_{\Tt^d}(D_pH)_\xi Dw\rho\,dx\,ds
$.
From \eqref{eq:div} we have
$(D_pH)_\xi\in L^p(\Tt^d\times\left[0,T\right])$ for any $p<\infty$.
Let $r>d$ and $q$ be the conjugate exponent, then
\begin{align*}
\int_0^t\int_{\Tt^d}(D_pH)_\xi Dw\rho &\le
\int_0^t\|(D_pH)_\xi\|_{2r}\|\sqrt{\rho}\|_{2q}\|Dw\sqrt{\rho}\|_2\\
&\le  \|(D_pH)_\xi\|_{L^4(0,t;L^{2r}(\Tt^d))} \|\rho\|^{\frac 12}_{L^2(0,t;L^q(\Tt^d))}
\Bigl(\int_0^t\int_{\Tt^d}|Dw|^2 \rho\Bigr)^{\frac 12}.
\end{align*}
This is clearly bounded by the bounds on $(D_pH)_\xi$.
\end{proof}

\section{Limit as $\epsilon \to 0$}\label{sec17}

In this section we investigate the behavior of the approximated solutions $(u^\epsilon,m^\epsilon)$ to \eqref{eq:smfg} as $\epsilon\to 0$. 

\begin{Lemma}\label{lemma01sec17}
Let $(u^\epsilon,m^\epsilon)$ be a solution of \eqref{eq:smfg}. Suppose that A\ref{ah}-\ref{alphaimp} hold. Then there exist $\ga\in\left(0,1\right)$ and $C>0$ such that \[\left\|u^\epsilon\right\|_{\mathcal{C}^{0,\ga}\left(\Tt^d\times\left[0,T\right]\right)}\leq C, \]uniformly in $\epsilon$. Furthermore  there exists $u\in \mathcal{C}^{0,\ga}\left(\Tt^d\times\left[0,T\right]\right)$, for some $\ga\in\left(0,1\right)$, such that $u^\epsilon\to u$ through some sub-sequence, in $\mathcal{C}^{0,\ga}\left(\Tt^d\times\left[0,T\right]\right)$.
\end{Lemma}
\begin{proof}
Since $D_xu^\ep\in L^\infty\left(\Tt^d\times\left[0,T\right]\right)$ we have
$
D_xu^\ep\in L^p\left(\Tt^d\times\left[0,T\right]\right)
$, for every $p<\infty$.
Also, Corollary \ref{freg} yields $
\left\|u^\epsilon_t\right\|_{L^p\left(\Tt^d\times\left[0,T\right]\right)}\leq C
$, for every $p<\infty$.
Then Morrey's inequality implies the result. The convergence follows from Ascoli-Arzela theorem. 
\end{proof}

\begin{Lemma}\label{lemma1sec17}
Let $(u^\epsilon,m^\epsilon)$ be a solution of \eqref{eq:smfg}. Suppose that A\ref{ah}-\ref{alphaimp} hold. Then, there exists $m\in \mathcal{C}^{0,\ga}\left(\Tt^d\times\left[0,T\right]\right)$, for some $\ga\in\left(0,1\right)$, such that $m^\epsilon\to m$ through some sub-sequence
in $\mathcal{C}^{0,\ga}\left(\Tt^d\times\left[0,T\right]\right)$.
\end{Lemma}
\begin{proof}
We note that because of Corollary \ref{mholder} we have $m^\epsilon\in\mathcal{C}^{0,\ga}\left(\Tt^d\times\left[0,T\right]\right)$, uniformly in $\epsilon$. Then, the family $(m^\epsilon)\epsilon$ is equicontinuous and equibounded. Hence, by Ascoli-Arzela theorem, 
 $m^\epsilon$ converges to some function $m$ as $\epsilon\to 0$, through a sub-sequence if necessary, and the limit $m$ satisfies the same uniform estimates as $m^\epsilon$. Since $m^\epsilon$ is uniform continuous, the result follows. 
\end{proof}

\begin{Corollary}\label{cor1sec17}
Let $(u^\epsilon,m^\epsilon)$ be a solution of \eqref{eq:smfg}. Suppose that A\ref{ah}-\ref{alphaimp} hold. Hence, the limit of $m^\epsilon$ as $\epsilon\to 0$ is a weak solution of \[m_t-\div\left(D_pH(x,Du)m\right)=\Delta m.\]
\end{Corollary}
\begin{proof}
We start by noticing that because of Corollary \eqref{uregular} we have $Du^\epsilon$ is pre-compact in, for instance, $L^2([0,T], L^2(\Tt^d))$. 
Therefore, through a sub-sequence,  
$Du^\epsilon\to Du$
both in $L^2([0,T], L^2(\Tt^d))$ and almost everywhere. 
 Let $\phi\in C^\infty_c\left((0,T)\times \Tt^d\right)$ be arbitrary. Multiply the equation for $m^\epsilon$ by $\phi$ and integrate to obtain
\[
0=\int_{0}^{T}\int_{\Tt^d}\left(m_t^\epsilon-\div(D_pHm^\epsilon)-\Delta m^\epsilon\right)\phi dxdt.
\] Integrating by parts and taking limits one obtains
\begin{align*}
0&=\lim_{\epsilon\to 0}\left(\int_{0}^{T}\int_{\Tt^d}-\phi_t m^\epsilon+D_pHD\phi m^\epsilon-\Delta\phi m^\epsilon dxdt\right)\\&=\int_{0}^{T}\int_{\Tt^d}\left(-\phi_t+D_pHD\phi-\Delta\phi\right) m dxdt,
\end{align*}where the second equality follows from Lemmas \ref{lemma1sec17} and \ref{lemma01sec17}. This concludes the proof.
\end{proof}

Next we present the proof of Theorem \ref{teo:intro1}. 

\begin{proof}[Proof of Theorems \ref{teo:intro1}] 
We firstly notice that Lemma \ref{lemma01sec17} and Corollary \ref{mholder} ensure that $u^\ep$ and $m^\ep$ are H\"older continuous, uniformly in $\epsilon$. Furthermore, because of Lemmas \ref{lemma1sec17} and \ref{lemma01sec17}, we have that $u^\ep\to u$ in $\mathcal{C}^{0,\ga}(\Tt^d\times\left[0,T\right])$ as well as $m^\ep\to m$ in $\mathcal{C}^{0,\ga}(\Tt^d\times\left[0,T\right])$, as $\ep\to 0 $.

From Corollary \ref{cor1sec17} it follows that $m$ is a weak solution to 
\[
m_t-\div(D_pH(x,Du)m)=\Delta m.
\]
Because $u^\ep$ is Lipschitz and $m^\ep$ converges uniformly,  
by combining Lemma \ref{lemma01sec17}
with Lemma \ref{lemma1sec17}, 
it follows that $u$ is a viscosity solution to
\[
-u_t+H(x,Du)=\Delta u+g(m).
\]
From the regularity results that we obtained for $(u^\epsilon, m^\epsilon)$, a simple bootstrapping argument by differentiating repeatedly the 
equations and using parabolic regularity ensures that $(u^\epsilon, m^\epsilon)$ satisfy uniform bounds in any Sobolev space. 
By noticing that $(u,m)$ inherits the regularity of $(u^\ep,m^\ep)$, one concludes the proof. 
\end{proof}

\bibliographystyle{alpha}
\bibliography{paper1}

\end{document}